\documentclass{amsart}
\usepackage{amssymb}
\usepackage{amscd}
\usepackage{amsmath}
\usepackage{epsfig}
\usepackage{graphicx}
\usepackage{color}

\usepackage{time}

\usepackage[all]{xy}

\usepackage{pgf,tikz}
\usepackage{caption}
\usepackage{shuffle}

\usetikzlibrary{cd} 

\usepackage{ulem}
\newcommand{\Add}[1]{\textcolor{black}{#1}}

\title[$p$-adic hypergeometric function]
{$p$-adic hypergeometric function related with $p$-adic multiple polylogarithms}
\author{Hidekazu Furusho}
\address{Graduate School of Mathematics, Nagoya University, Chikusa-ku, Furo-cho, Nagoya, 464-8602,  Japan}
\email{furusho@math.nagoya-u.ac.jp}


\date{September 23, 2025}
\newtheorem{thm}{Theorem}
\newtheorem{lem}[thm]{Lemma}

\newtheorem{prop}[thm]{Proposition}

\theoremstyle{remark}

\theoremstyle{definition}
\newtheorem{defn}[thm]{Definition}
\newtheorem{rem}[thm]{Remark}

\newtheorem{eg}[thm]{Example}       

\numberwithin{equation}{section}
\numberwithin{figure}{section}

\newcommand{\GL}{\mathrm{GL}}

\newcommand{\KZ}{\mathrm{KZ}}

\newcommand{\col}{\mathrm{Col}}

\newcommand{\Q}{\mathbb{Q}}
\newcommand{\C}{\mathbb{C}}
\newcommand{\R}{\mathbb{R}}
\newcommand{\Z}{\mathbb{Z}}
\newcommand{\N}{\mathbb{N}}

\newcommand{\F}{\mathbb{F}}

\newcommand{\V}{\mathcal{V}}
\newcommand{\X}{\mathcal{X}}
\newcommand{\PP}{\mathcal{R}}

\newcommand{\aaa}{\mathsf a}
\newcommand{\bbb}{\mathsf b}
\newcommand{\ccc}{\mathsf c}
\newcommand{\uuu}{\mathsf u}
\newcommand{\vvv}{\mathsf v}

\newcommand{\e}{\mathsf e}

\newcommand{\zzz}{\mathsf z}

\newcommand{\LL}{\mathsf{L}}

\newcommand{\wt}{\mathrm{wt}}

\newcommand{\Li}{\mathrm{Li}}

\newcommand{\cok}{\mathrm{cok}}

\newcommand{\red}{\mathrm{red}}
\newcommand{\loc}{\mathrm{loc}}

\newcommand{\Mat}{\mathrm{Mat}}

\newcommand*{\GGamma}[3]{\Gamma_{#1}{\left({{#2}\atop{#3}}\right)}}

\newcommand*{\uHG}[4]{{}_{2}F_{1}\!{}\strut^{#1}{\left({{#2}\atop{#3}}\middle|{#4}\right)}}

\catcode`,\active

\catcode`\,12

\begin{document}
\bibliographystyle{amsalpha+}
\maketitle

\begin{abstract}
This paper introduces a $p$-adic analogue of Gauss's hypergeometric function,
constructed via a method that is distinct from
distinct from Dwork's approach. 
The idea of our construction
is motivated by the Ohno-Zagier formula,
which is elucidated through
the relationship between the hypergeometric differential equation and the Knizhnik-Zamolodchikov (KZ) equation.
We develop a rigorous framework for the residue-wise analytic prolongation of our 
$p$-adic hypergeometric function
by exploring its relationship with
$p$-adic multiple polylogarithms.
Through a detailed analysis of its local behavior near the point $1$,
we show a $p$-adic version of Gauss hypergeometric theorem
for the function.
\end{abstract}

\tableofcontents
\setcounter{section}{-1}
\section{Introduction}\label{introduction}
Let $p$ be a prime number. This paper investigates a $p$-adic analogue of
Gauss's hypergeometric function.
By exploiting  techniques derived from the $p$-adic Knizhnik-Zamolodchikov (KZ) equation,
we introduce a formal version 
of $p$-adic hypergeometric function 
(in Definition \ref{defn: formal version of pHG})
denoted as
\begin{equation}\label{eq: formal pHG}
\uHG{p,\varpi}{\aaa,\bbb}{\ccc}{-}:\Q_p\setminus\{1\}\to \Q_p[[\aaa,\bbb,\ccc-1]].
\end{equation}
where $\varpi$ is a branch of $p$-adic logarithm (see \S \ref{subsec: review on Coleman integrations}).
This function satisfies the
hypergeometric differential equation
\begin{equation}\label{eq:formal HG diff eq}
z(1-z)\frac{d^2w}{dz^2}+\{\ccc-(\aaa+\bbb+1)z\}\frac{dw}{dz}-\aaa\bbb w=0.
\end{equation}
with $w=\uHG{p,\varpi}{\aaa,\bbb}{\ccc}{z}$.
This function, formally defined in Definition \ref{defn: formal version of pHG},
takes values in formal power series. 
And we provide a rigorous framework for interpreting this object under specialization of the parameters
$(\aaa,\bbb,\ccc)\mapsto (\alpha,\beta,\gamma)\in\Q_p^3$.
Our main result is summarized  as follows: 

\begin{thm}\label{thm:main_result}
Assume that  $\alpha,\beta,\gamma-1\in p\Z_{(p)}$ 
with 
$\Z_{(p)}:=\Z_p\cap\Q$,
$\gamma\not\in \Z_{\leq 0}$,
$\alpha+\beta-\gamma, 
\alpha-\beta\in \Z_{(p)}\setminus\N_{\pm}$ and
let $\varpi$ be a branch of the $p$-adic logarithm 
chosen such that $\varpi\in\Z_p$.
Then the following assertions hold within a rigorously constructed framework of specialization from the formal parameters
$(\aaa,\bbb,\ccc)$ to the values $ (\alpha,\beta,\gamma)$:

(1). 
There exists a well-defined  map
$$
\uHG{p,\varpi}{\alpha,\beta}{\gamma}{-}:\Q_p\setminus\{1\}\to \Q_p
$$
which satisfies the hypergeometric differential equation
\eqref{eq:formal HG diff eq} with
$(\aaa,\bbb,\ccc)=(\alpha,\beta,\gamma)$.

(2). 
The restriction of this map to the open unit disk centered at $z=0$
coincides with the convergent series
$\sum_{n=0}^\infty\frac{(\alpha)_n(\beta)_n}{(\gamma)_n n!}z^n$.
Here $(s)_n$ denotes the Pochhammer symbol, defined as 
$(s)_0=1$, $(s)_n=s(s+1)\cdots (s+n-1)$ for $n > 0$.

(3).
The restriction of this map to $\Z_p\setminus \{1+p\Z_p\}$
is independent of the choice of the branch
$\varpi$ of the $p$-adic logarithm.

(4). The $p$-adic Gauss hypergeometric theorem holds at $z=1$, that is,
$$
\lim_{\substack{ z\to 1}}\uHG{p,\varpi}{\alpha,\beta}{\gamma}{z}
=\prod_{k=1}^\infty\frac{\Gamma_p(1+p^k\mu+p^k\alpha)\cdot\Gamma_p(1+p^k\mu+p^k\beta)}
{\Gamma_p(1+p^k{\mu})\cdot\Gamma_p(1+p^k{\nu})}
$$
where $\mu=1-\gamma$ and $\nu=\alpha+\beta+1-\gamma$ and
$\Gamma_p$ denotes Morita's $p$-adic gamma function
(cf. \eqref{eq: Morita gamma}).
\end{thm}

While the  case $(\alpha,\beta,\gamma)=(\frac{1}{2},\frac{1}{2},1)$,
corresponding to elliptic integrals,
falls outside the scope of Theorem \ref{thm:main_result}, we establish the following result for a broader class of parameters
that includes this case.

\begin{thm} \label{thm:extended_convergence}
Assume that $\alpha,\beta \in \Z_p$
and $\gamma\in \Z_{(p)}\setminus \Z_{\leqslant 0}$.
Then there is a rigorous framework for the specialization
$(\aaa,\bbb,\ccc)\mapsto (\alpha,\beta,\gamma)$,
under which the resulting map
$$
\uHG{p}{\alpha,\beta}{\gamma}{-}:\Z_p\setminus \{1+p\Z_p\} \to \Q_p.
$$
is well-defined and independent of the choice of
branch parameter $\varpi$ of the $p$-adic logarithm.
In addition under the assumption of Theorem \ref{thm:main_result},
our map coincides with map in Theorem \ref{thm:main_result}.
\end{thm}

This theorem significantly expands the domain of convergence for our $p$-adic hypergeometric function, encompassing important special cases such as the Legendre function.
It demonstrates the robustness of our construction and its applicability to a wider range of parameters than initially considered.

Our results are, in fact, more technically sophisticated and proven in a more general setting. 
The conditions for $\alpha, \beta, \gamma$ are relaxed, 
$\Q_p$ can be replaced with a larger field.
The core of our proof strategy lies in examining the convergence residue-wise by exploiting a $p$-adic analogue of Ohno-Zagier formula \eqref{eq: expansion of HG}
which connects the function with $p$-adic multiple polylogarithms
and by establishing connection formulas of the function \eqref{eq: formal pHG}
---between $0$ and $\infty$ in Theorem \ref{thm: formal connection formula between 0 and infinity}
 and between $0$ and $1$ in Theorem \ref{thm: connection between 0 and 1}.

The proof structure is as follows:
\begin{enumerate}
\renewcommand{\labelenumi}{(\roman{enumi})}
\item
A rigorous framework for specialization, convergence, and the coincidence with the classical power series on the residue disk
$]0[$ is established in \S\ref{subsec: local behavior at 0}.
\item
Rigorous specialization on the residue disk $]\infty[$ is establoshed in \S \ref{subsec: local behavior at infty}.
\item
Rigorous specialization on the residue disk $]1[$, together with the proof of 
the $p$-adic Gauss hypergeometric theorem are presented in \S \ref{subsec: local behavior at 1}.
\item
Convergence on the remaining residue disks is addressed in \S \ref{subsec: prolongation on the rest disks}.
\end{enumerate}


Our $p$-adic hypergeometric function  can be regarded as a $p$-adic counterpart 
to the $\ell$-adic hypergeometric function introduced in our previous paper \cite{F21}.
It is worth noting that another $p$-adic analogue of the hypergeometric function has been extensively studied by Dwork \cite{Dw}, 
whose region of definition is smaller than that of ours.
However, the region of definition for Dwork's function is more restricted than ours. The precise relationship between our function and Dwork's remains an open question for further investigation.

{\it Acknowledgments.}
The author has been supported by grants JSPS KAKENHI  JP18H01110, JP20H00115,
JP21H00969 and JP21H04430.
He is grateful to 
Shinichi Kobayashi
and
Seidai Yasuda
for valuable discussions.

\section{Complex case}\label{sec: complex case}
This section recalls the relationship between KZ equation and Euler's hypergeometric differential equation,  and expounds
upon the Ohno-Zagier formula,  which connect
the hypergeometric function with multiple polylogarithms.

\subsection{Multiple polylogarithms}\label{subsec: MPL}
The (formal) {\it KZ (Knizhnik-Zamolodchikov) equation} 
is a differential equation over $\mathcal X(\C)$,
where $\mathcal{X} = \mathbb{P}^1 \setminus \{0,1,\infty\}$.
It takes the form:
\begin{equation}\label{eq:KZ-equation}
\frac{d}{dz}G(z)=\{\frac{\e_0}{z}+\frac{\e_1}{z-1}\}\cdot G(z)
\end{equation}

Here
$ G(z)=G^\C(z)\in {\mathcal O}_{\tilde{\mathcal X}}\langle\langle \e_0,\e_1\rangle\rangle$
is analytic in complex variables, with each coefficient being analytic
over the universal unramified covering ${\tilde{\mathcal X}}$,
and takes values in the noncommutative formal power series ring
$\C\langle\langle \e_0,\e_1\rangle\rangle$.

The KZ equation exhibits regular Fuchsian singularities at
$z=0$, $1$ and $\infty$

In \cite[\S 3]{Dr89}, Drinfeld  considers the fundamental solution 
$G^\C_{\vec{01}}(\e_0,\e_1)(z)$
which is uniquely determined by
the asymptotic property
$G^\C_{\vec{01}}(\e_0,\e_1)(z)z^{-\e_0}\to 1$ when $z\in\R_+$ approaches $0$.
Here $z^{-\e_0}$ is defined as the formal series $\sum_{i=0}^\infty\frac{(-\log z)^i}{i!}\e_0^i$.

He further investigates 5 other solutions of the  KZ equation
with certain specific asymptotic properties,
which are described as
\begin{align}\label{eqB}
G^\C_{\vec{10}}(\e_0,\e_1)(z)
&=G^\C_{\vec{01}}(\e_1,\e_0)(1-z)
, \\ \notag
G^\C_{\vec{1\infty}}(\e_0,\e_1)(z)
&=G^\C_{\vec{01}}(\e_1,\e_\infty)(1-\frac{1}{z})
, \\ \notag
G^\C_{\vec{\infty 1}}(\e_0,\e_1)(z)
&=G^\C_{\vec{01}}(\e_\infty,\e_1)(\frac{1}{z})
, \\ \notag
G^\C_{\vec{\infty 0}}(\e_0,\e_1)(z)
&=G^\C_{\vec{01}}(\e_\infty,\e_0)(\frac{1}{1-z})
, \\ \notag
G^\C_{\vec{0\infty}}(\e_0,\e_1)(z)
&=G^\C_{\vec{01}}(\e_0,\e_\infty)(\frac{z}{z-1})
\end{align}
with $\e_\infty=-\e_0-\e_1$.
The quotient
$$
\varPhi^\C_\KZ(\e_0,\e_1):=
G^\C_{\vec{10}}(\e_0,\e_1)(z)^{-1}\cdot G^\C_{\vec{01}}(\e_0,\e_1)(z)
\in \C\langle\langle \e_0,\e_1\rangle\rangle
$$
is independent of $z$ and is called
the {\it KZ associator}.
We also note that
$\exp\{\pi\sqrt{-1}\e_0\}=G^\C_{\vec{01}}(\e_0,\e_1)(z)^{-1}\cdot G^\C_{\vec{0\infty}}(\e_0,\e_1)(z)$

Let $\mathbb A=
\C\langle \e_0,\e_1\rangle$ 
be the polynomial part of $\C\langle\langle \e_0,\e_1\rangle\rangle$.
We call an element of $\mathbb A$ which is a monomial 
with coefficient $1$
by a {\it word}.
But exceptionally we shall not call $1$ a word.
For each word $W$, its {\it weight} $\wt (W)$
(resp. its {\it depth}  $\mathrm{dp}(W)$)
is defined to be 
the sum of exponents of $\e_0$ and $\e_1$ (resp. of $\e_1$) in $W$.
Let $\mathbb M'=\mathbb A\cdot \e_1=\{ F\cdot \e_1 \bigm| F\in\mathbb A\}$
be the $\C$-linear subspace of 
$\mathbb A$.
Note that there is a natural $\C$-linear surjection from
$\mathbb A$  to
$\mathbb A\bigm/\mathbb A\cdot \e_0$.
By identifying the latter space with 
$\C\cdot 1+\mathbb M' (=\C\cdot 1+\mathbb A\cdot \e_1)$,
we obtain the $\Q$-linear map
$f':\mathbb A\twoheadrightarrow\mathbb A \bigm/\mathbb A\cdot \e_0\overset{\sim}{\to}\C\cdot 1+\mathbb M'\hookrightarrow\mathbb A$.
For each word $W=\e_1^{q_0}\e_0^{p_1}\e_1^{q_1}\e_0^{p_2}\e_1^{q_2}\dotsm \e_0^{p_k}\e_1^{q_k}$
($k\geqslant 0$, $q_0\geqslant 0$, $p_i,q_i\geqslant 1$ for $i\geqslant 1$)
in $\mathbb M'$, we define
$$
\Li^\C_W(z):=\Li^\C_{
\underbrace{1,\ldots 1}_{q_k-1},
p_k+1,
\underbrace{1,\ldots 1}_{q_{k-1}-1},
p_{k-1}+1,\ldots\ldots ,1,p_1+1
\underbrace{1,\ldots 1}_{q_0}
}(z).
$$
Here $\Li^\C_{\mathbf k}(z)\in {\mathcal O}_{\tilde{\mathcal X}}$ for a tuple ${\mathbf k}=(k_1,\dots,k_n)\in\N^n$ ($n\in\N$)
is the {\it multiple polylogarithm}. 
This is a complex-valued function whose restriction to the open unit disk centered at the origin is defined by the formal power series
\begin{equation}\label{eq:MPL}
\Li^\Q_{\mathbf k}(z):=\sum_{0<m_1<\cdots<m_n}\frac{z^{m_n}}{m_1^{k_1}\cdots m_n^{k_n}}
\in\Q[[z]]
\end{equation} 
and satisfies the following recursive differential equation:
\begin{align}\label{eq: diff eq MPL}
&\frac{d}{dz}\Li^\C_{k_1,\dots,k_n}(z)
=\begin{cases}
\frac{1}{z}\Li^\C_{k_1,\dots,k_n-1}(z) & (k_n>1), \\
\frac{1}{1-z}\Li^\C_{k_1,\dots,k_{n-1}}(z) & (k_n=1),
\end{cases}
&\frac{d}{dz}\Li_1^\C(z)=\frac{1}{1-z} .
\end{align}
which facilitates its analytic continuation beyond the unit disk.

By extending linearly, we get the $\C$-linear map
$\Li^\C(z):\mathbb M'\to {\mathcal O}_{\tilde{\mathcal X}}$
which sends each word $W$ in $\mathbb M'$ to $\Li^\C_W(z)$.

It is shown that coefficients of
$G^\C_{\vec{01}}(\e_0,\e_1)(z)$ are given by polynomial combinations of
multiple polylogarithms
and the logarithm $\log(z)$.

\begin{prop}[\cite{F03,F04}]\label{prop:explicit formulae}
Let $G^\C_{\vec{01}}(\e_0,\e_1)(z)=1+\sum\limits_{W:\mathrm{words}} J(W)(z) \ W $.
Then each coefficient $ J(W)(z)$ is expressed as follows:
\begin{enumerate}
\renewcommand{\labelenumi}{(\alph{enumi})}
\item  When $W$ is in $M'$, \ \ \ $J(W)(z)=(-1)^{dp (W)}\Li^\C_W(z)$.
\item When $W$ is written as $V\e_0^r$ ($r\geqslant 0,V\in M'$),
\[
J(W)(z)=\underset{0\leqslant s,0\leqslant t}{\sum_{s+t=r}}
(-1)^{dp (W)+s}\Li^\C_{f'(V\shuffle \e_0^s)}(z)
\frac{\{\log(z)\}^t}{t!}\quad .
\]
\item When $W$ is written as 
$\e_0^r$ ($r\geqslant 0$), \ \ \ $J(W)(z)=\frac{\{\log(z)\}^r}{r!}$.
\end{enumerate}
\end{prop}

Here $\shuffle$ stands for the shuffle product 
arising from the product of Chen's path iterated integrals. 

We note that a similar type of formula describing each coefficient of the KZ-associator
as a linear combination of multiple zeta values
$$\zeta({\mathbf k}):=\sum_{0<m_1<\cdots<m_n}\frac{1}{m_1^{k_1}\cdots m_n^{k_n}}\in\R$$
which are  limit values of multiple polylogarithm to $1$ for $k_n>1$,
is given in \cite{F03}.

\subsection{Hypergeometric function}\label{subsec: Hypergeometric function}
{\it Gauss's hypergeometric function} (consult \cite{AAR} for example)
is a complex analytic function $\uHG{\C}{a,b}{c}{z}\in {\mathcal O}_{\tilde{\mathcal X}}$
whose restriction to the open unit disk centered at the origin is given by the formal power series
\begin{equation}\label{def eq: HG}
\uHG{\Q}{a,b}{c}{z}:=\sum_{n=0}^\infty\frac{(a)_n(b)_n}{(c)_n n!}z^n
\in\Q(a,b,c)[[z]].
\end{equation}
This series  converges for $|z|<1$,
where $a,b,c$ are complex numbers with $c\neq 0,-1,-2,\dots$. 


The hypergeometric function has been a subject of extensive study for several centuries. 
One of its most celebrated results is the hypergeometric theorem
(due to Gauss):

\begin{equation}\label{eq: hypergeometric equation}
\uHG{\C}{a,b}{c}{1}=\frac{\Gamma(c)\Gamma(c-a-b)}{\Gamma(c-a)\Gamma(c-b)},
\qquad
\text{ when }
\Re(c)>\Re(a+b).
\end{equation}

The hypergeometric function satisfies 
{\it Euler's hypergeometric differential equation}:
\begin{equation}\label{eq: EHG diff eq}
z(1-z)\frac{d^2w}{dz^2}+\{c-(a+b+1)z\}\frac{dw}{dz}-abw=0.
\end{equation}
This equation facilitates the analytic continuation of the function along topological paths
starting from $0$ to any $z$ in $\X(\C)$.

Put
$${X_0}=
\begin{pmatrix}
0 & b \\
0 & u
\end{pmatrix} 
\text{ and }
{Y_0}=
\begin{pmatrix}
0 & 0 \\
a & v
\end{pmatrix}
\in\Mat_2(\C)
$$ 
with
$u=1-c$ and $v=a+b+1-c$. 
The hypergeometric differential equation can be reformulated as 
a system of first-order equations,
the following KZ-like differential equation:
\begin{equation*}
\frac{d}{dz}\vec{g}_\C
=\left\{\frac{1}{z}X_0+\frac{1}{1-z}Y_0\right\}\cdot \vec{g}_\C
\end{equation*}
with
$\vec{g}_\C=\vec{g}_\C(w)=
\begin{pmatrix}
w \\ \frac{z}{b}\frac{dw}{dz}
\end{pmatrix}\in ({\mathcal O}_{\tilde{\mathcal X}})^{\oplus 2}$
when $b\neq 0$.

Conversely consider the substitution $(X_0,-Y_0)$ for $(\e_0,\e_1)$ 
in a solution $G^\C(\e_0,\e_1)(z)$ of the KZ-equation \eqref{eq:KZ-equation}.
When convergent, this substitution yields two solutions $w_1$, $w_2$ of
the hypergeometric differential equation \eqref{eq: EHG diff eq},
which correspond to the $(1,1)$ and $(1,2)$-entries of $G^\C(X_0,-Y_0)(z)$
respectively.
These solutions satisfy:
$$(\vec g_\C(w_1), \vec g_\C(w_2))=G^\C(X_0,-Y_0)(z)\in\Mat_2({\mathcal O}_{\tilde{\mathcal X}}).
$$
This connection between the KZ equation and the hypergeometric differential equation is elaborated in \cite{O}. 
Define the following matrices:
\begin{align}\label{eqC}
&\V_{\vec{01}}^\C(z):=
G^\C_{\vec{01}}(X_0,-Y_0)(z)\cdot
\begin{pmatrix}
1 & 1 \\
0 & \frac{u}{b}
\end{pmatrix}
, \\ 
\notag
&\V^\C_{\vec{10}}(z):=
G^\C_{\vec{10}}(X_0,-Y_0)(z)\cdot
\begin{pmatrix}
1 & 0 \\
\frac{-a}{v} & \frac{v-1}{b}
\end{pmatrix}
, \\ 
\notag
&\V^\C_{\vec{1\infty}}(z):=
G^\C_{\vec{1\infty}}(X_0,-Y_0)(z)\cdot
\begin{pmatrix}
1 & 0 \\
\frac{-a}{v} & \Add{\frac{1-v}{b}}
\end{pmatrix}
, \\ 
\notag
&\V^\C_{\vec{\infty 1}}(z):=
G^\C_{\vec{\infty 1}}(X_0,-Y_0)(z)\cdot
\begin{pmatrix}
1 & \Add{1} \\
\frac{-a}{b} & -1
\end{pmatrix}
, \\ 
\notag
&\V^\C_{\vec{\infty 0}}(z):=
G^\C_{\vec{\infty 0}}(X_0,-Y_0)(z)\cdot
\begin{pmatrix}
1 & 1 \\
\frac{-a}{b} & -1
\end{pmatrix},\\
\notag
&\V^\C_{\vec{0\infty}}(z):=
G^\C_{\vec{0\infty}}(X_0,-Y_0)(z)\cdot
\begin{pmatrix}
1 & 1 \\
0 & \frac{u}{b}
\end{pmatrix}
.
\end{align}
These matrices yield scalar multiples of half of Kummer's 24 solutions (cf. \cite{F21}):
\begin{align}\label{eqD}
\begin{pmatrix} 1, & 0 \end{pmatrix}\cdot \V^\C_{\vec{01}}(z) &=
\begin{pmatrix} \uHG{\C}{a,b}{c}{z}, & z^{1-c}\uHG{\C}{b+1-c,a+1-c}{2-c}{z}\end{pmatrix}, \\
\notag
\begin{pmatrix} 1, & 0 \end{pmatrix}\cdot \V^\C_{\vec{10}}(z)
&=
\begin{pmatrix} \uHG{\C}{a,b}{a+b+1-c}{1-z}, & (1-z)^{c-a-b}\uHG{\C}{c-a,c-b}{c-a-b+1}{1-z}\end{pmatrix}, \\
\notag
\begin{pmatrix} 1, & 0 \end{pmatrix}\cdot \V^\C_{\vec{1\infty}}(z)
&=
\begin{pmatrix} z^{-a}\uHG{\C}{a,a+1-c}{a+b-c+1}{1-\frac{1}{z}}, & 
z^{b-c}(z-1)^{c-a-b}\uHG{\C}{1-b,c-b}{1-a-b+c}{1-\frac{1}{z}}
\end{pmatrix}, \\
\notag
\begin{pmatrix} 1, & 0 \end{pmatrix}\cdot \V^\C_{\vec{\infty 1}}(z)
&=
\begin{pmatrix} z^{-a}\uHG{\C}{a,a+1-c}{a-b+1}{\frac{1}{z}}, & 
z^{-b}\uHG{\C}{b+1-c,b}{b-a+1}{\frac{1}{z}}
\end{pmatrix}, \\
\notag
\begin{pmatrix} 1, & 0 \end{pmatrix}\cdot \V^\C_{\vec{\infty 0}}(z)
&=
\begin{pmatrix} (1-z)^{-a}\uHG{\C}{a,c-b}{a-b+1}{\frac{1}{1-z}},& 
(1-z)^{-b}\uHG{\C}{c-a,b}{1-a+b}{\frac{1}{1-z}}
\end{pmatrix},\\
\notag
\begin{pmatrix} 1, & 0 \end{pmatrix}\cdot \V^\C_{\vec{0\infty}}(z)
&=
\begin{pmatrix} (1-z)^{-a}\uHG{\C}{a,c-b}{c}{\frac{z}{z-1}}, & 
(1-z)^{-a}(\frac{z}{z-1})^{1-c}\uHG{\C}{1-b,a-c+1}{2-c}{\frac{z}{z-1}}
\end{pmatrix}
\end{align}
where we consider the principal branch under appropriate conditions for $a,b,c$.
The remaining half of the 24 solutions can be obtained via 
Euler's
transformation formula
$\uHG{\C}{a,b}{c}{z}=
(1-z)^{c-a-b}\uHG{\C}{c-a,c-b}{c}{z}$.


Specifically, we have
\begin{equation}\label{eq:HG=11}
 \uHG{\C}{a,b}{c}{z}=[\V_{\vec{01}}^\C(z)]_{(1,1)}=
 [G^\C_{\vec{01}}(X_0,-Y_0)(z)]_{(1,1)}
\end{equation}
where the right lower suffix $(1,1)$ means the $(1,1)$-entry of the matrix.

The following theorem, due to Ohno and Zagier, establishes a profound connection between hypergeometric functions and multiple polylogarithms, providing a powerful tool for analyzing these special functions.

\begin{thm}[\cite{OZ}]\label{thm:OZ}
The hypergeometric function can be expressed in terms of multiple polylogarithms as follows:
\begin{equation}\label{eq: expansion of HG}
\uHG{\C}{a,b}{c}{z}=
1+ ab\sum_{\substack{k,n,s> 0 \\k\geqslant n+s, \ n \geqslant s}}
g^\C_0(k,n,s)(z)u^{k-n-s}v^{n-s}
(ab+uv)^{s-1} 
\end{equation}
with
$u=1-c$, $v=a+b+1-c=a+b+u$ and
\begin{equation*}\label{eq: g0}
g^\C_0(k,n,s)(z)=
\sum_{\substack{\wt({\mathbf k})=k, \ \mathrm{dp}({\mathbf k})=n, \ \mathrm{ht}({\mathbf k})=s \\ {\mathbf k}:\text{ admissible index}}} \Li^\C_{\mathbf k}(z).
\end{equation*}
Here an {\it admissible index} refers  to a tuple ${\mathbf k}=(k_1,\dots,k_n)\in\N^n$ ($n\in\N$) 
with $k_n>1$ and 
$\wt({\mathbf k})=k_1+\cdots+k_n$, 
$\mathrm{dp}({\mathbf k})=n$, 
$\mathrm{ht}({\mathbf k})=\sharp\{i\bigm| k_i>1\}$.
\end{thm}

Shu Oi \cite[Theorem 3.1]{O} rediscovered the formula
by combining  Proposition \ref{prop:explicit formulae}
together with Equation \eqref{eq:HG=11}.

\section{$p$-adic case}\label{sec: p-adic case}
We will introduce the $p$-adic hypergeometric function
by a $p$-adic analogue of \eqref{eq:HG=11}.
By exploiting a $p$-adic analogue of \eqref{eq: expansion of HG},
we perform its  analytic prolongation residue-wise
and establish a $p$-adic analogue of Gauss hypergeometric theorem.
These results are then applied to prove Theorems \ref{thm:main_result} and \ref{thm:extended_convergence}.

\subsection{Review on the theory of Coleman integrations}
\label{subsec: review on Coleman integrations}
This subsection provides a concise review of Coleman functions
in the case for $\mathcal{X} = \mathbb{P}^1 \setminus \{0,1,\infty\}$.
which will be used in subsequent arguments.

Let  $\C_p$ be the field of 
$p$-adic complex numbers, which is a topologically complete and algebraically closed extension of the field of $p$-adic numbers $\Q_p$,
equipped with the $p$-adic valuation $|\cdot|_p$.
It plays a role analogous to that of the complex numbers in classical analysis.
Let $\red:{\mathbb P}^1(\C_p)\to{\mathbb P}^1(\overline{\F_p})$ be the reduction map.
where $\overline{\F_p}$ is the algebraic closure of the finite field $\F_p$.
For each subset  $S\subset{\mathbb P}^1(\overline{\F_p})$,
we define the residue class associated with $S$ as
$]S[:=\red^{-1}(S)$.

A {\it $p$-adic logarithm}  is defined as a group homomorphism 
from the multiplicative group $\C_p^\times$ to the additive group $\C_p$.
It admits the usual Taylor expansion $\sum_{k=1}^\infty\frac{(-1)^{k+1}(z-1)^k}{k}$ on the residue class $]1[$.
Here there exists an isomorphism  $\C_p\simeq p^\Q\times W\times ]1[$
where $W$ is the group of roots of unity whose order is prime to $p$
and $]1[$ corresponds to a subgroup of $\C_p^\times$.
The $p$-adic logarithm is uniquely determined outside the region on $]0[$ and $]\infty[$,
and is commonly denoted by $\log^p(z)$.
Otherwise, it is not uniquely determined unless one specifies the image of
$p$, denoted by $\varpi\in\C_p$.
In such cases, the associated  $p$-adic logarithm is written as 
$\log^{p,\varpi}(z)$, thus we have
$\log^{p,\varpi}(p)=\varpi$.

For $x\in{\mathcal X}(\overline{\F_p})$, 
let us choose a local parameter (an analytic isomorphism)
$z_x:]x[\overset{\simeq}{\to}]0[$.
Then the algebra $A(]x[)$ of {\it rigid analytic analytic functions} on $]x[$  (cf. \cite{BGR}) 
can then be identified with the algebra
$$
\{\sum_{m=0}^\infty c_mz_x^m\in\C_p[[z_x]]\mid |c_m|_p\lambda^n\to 0
\text{ for any } 0<\lambda<1\}.
$$
While for $x=0,1,\infty$, we consider a local parameter
$z_x:\ ]x[-\{x\} \overset{\simeq}{\to}\ ]0[-\{0\}$ and 
define the annuli $U_r:=\{z\in\C_p\mid r<|z|_p<1\}$ ($0\leqslant r<1$).
We then define the space:
$$
A_\varpi^{\log}(]x[):= \underset{r\to 1}{\mathrm{lim }}
A\Bigl(]x[ \ \cap z_x^{-1}(U_r)\Bigr)\Bigl[\log^{p,\varpi}(z_x)\Bigr]
$$
where 
$A\Bigl(]x[ \ \cap z_x^{-1}(U_r)\Bigr)$
denotes the algebra of functions converging on $]x[ \ \cap z_x^{-1}(U_r)$.
This algebra can be identified with the space
$$
\left\{\sum\limits_{m=-\infty}^{\infty}c_m\cdot z_x^m \in\C_p[[z_x^\pm]]
\Bigm|
\frac{|c_m|_p}{\lambda^m}\to 0 \text{ when  } m\to\pm\infty
\text{ for any } 0<\lambda<r 
\right\}.
$$
It is worth noting that 
$\log^{p,\varpi}(z_x)$ is a local analytic function on 
$]x[ - x$,  and remains transcendental over the direct limit
$\underset{r\to 1}{\mathrm{lim }}
A\Bigl(]x[ \ \cap z_x^{-1}(U_r)\Bigr)$.
Moreover, the definition of $A_\varpi^{\log}(]x[)$
does not depend on the specific choice of the local parameter $z_x$.
Put 
${\mathcal O}_\varpi^\loc=
\prod_{x\in{\mathcal X}}A(]x[)\oplus
\prod_{x\not\in{\mathcal X}}A_\varpi^{\log}(]x[)$
and
${\Omega}_\varpi^\loc=
\prod_{x\in{\mathcal X}}\Omega(]x[)\oplus
\prod_{x\not\in{\mathcal X}}\Omega_\varpi^{\log}(]x[)$
with 
$\Omega(]x[)={A}(]x[)dz_x$ and
$\Omega_\varpi^{\log}(]x[)={A}_\varpi^{\log}(]x[)dz_x$.
The componentwise differential yields a map
$
d^\loc: \mathcal O^\loc_\varpi\to  \Omega^\loc_\varpi
$
with 
$$
\ker d^\loc=\prod_{x\in {\mathbb P}^1(\overline{\F_p})}\C_p 
\qquad\text{ and }\qquad
\cok \ d^\loc=\{0\}.$$

Next, we consider a series of the form
\begin{equation*}\label{eq:MLdecomp}
f(z)=\sum_{m\geqslant 0}c_{1,m}z^{m}+\sum_{m> 0}\frac{c_{2,m}}{z^m}
+\sum_{m>0}\frac{c_{3,m}}{(z-1)^m}
\in\C_p[[z,z^{-1},{(z-1)}^{-1}]]
\end{equation*}
where the coefficients satisfy the condition:
$|c_{i,m}|_p\cdot \lambda^{-m}\to 0$
as $m\to \infty$ for $i=1,2,3$.
Such a series defines a function on a subset of ${\mathbb P}^1(\C_p)$ 
which extends beyond $]{\mathcal X}(\overline\F_p)[$.
These functions are referred to as {\it overconvergent functions}.
We denote the algebra of such overconvergent functions by $\mathcal O^\dag$.
Put $\Omega^\dag:={\mathcal O}^\dag dz$.
The differential
$
d^\dag: \mathcal O^\dag \to \Omega^\dag 
$
is defined with
$$
\ker d^\dag=\C_p
\qquad\quad\text{ and }\qquad
\dim_{\C_p} \cok \ d^\dag= 2 
$$
and is compatible with natural inclusions
$\mathcal O^\dag\subset \mathcal O_\varpi^\loc$  and
$\Omega^\dag\subset \mathcal O_\varpi^\loc$
obtained by residuewise restriction 
(cf. consult \cite{Ber} and related references.)

Finally, the algebra $\mathcal O^\col_\varpi$ of {\it Coleman functions}, attached to a branch $\varpi $, 
is equipped with the differential operator
$d^\col:\mathcal O^\col_\varpi\to \Omega^\col_\varpi:=\mathcal O^\col_\varpi dz$.
It is defined as the minimal intermediate subalgebra 
which makes  the following  diagram commutative
\[
\begin{tikzcd}
\mathcal O^\loc_\varpi\arrow[r, "d^\loc"] 
&  \Omega^\loc_\varpi  \\
\mathcal O^\col_\varpi \arrow[u, hook] \arrow[r, "d^{\col}"]
& \Omega^\col_\varpi  \arrow[u, hook] \\
\mathcal O^\dag \arrow[u, hook] \arrow[r, "d^\dag"]
& \Omega^\dag \arrow[u, hook].
\end{tikzcd}
 \]
and  satisfies the following properties
\begin{equation*}\label{eq:ker and cok of d Col}
\ker d^{\col}=\C_p
\qquad\quad\text{ and }\qquad
\cok \ d^{\col}=\{0\}.
\end{equation*}
The {\it Coleman integration} 
$
\int_{(\varpi)}:\ \Omega_\varpi^{\col} \to \mathcal O^\col_\varpi(X) / \C_p,
$
is defined as a section of $d^\col$ modulo constants.
Since the integration map is defined up to a global constant,
the expression
$\int_a^b f(z)dz:=F(b)-F(a)\in \mathcal O^\col_\varpi$
is well-defined for $a,b\in \mathbb P^1(\C_p)$ and $f\in \mathcal O^\col_\varpi$
where $F(z)$ is any lift of $\int f(z)dz$.
Whenever the values $F(a)$ and $F(b)$ are well-defined
for given points $a, b \in \mathbb P^1(\C_p)$,
we denote their difference $F(b)-F(a)$ as $\int_a^b f(z)dz$.
This notation ensures that the integral is independent of the choice of the primitive $F(z)$, as any two choices differ by a constant.
By construction, 
Coleman functions satisfy the following important properties 

\begin{prop}[\cite{Col}]
(i). {\bf Coincidence Principle:}
If $f\in \mathcal O^\col_{\varpi}$ vanishes on a residue class,
then  $f$ is identically $0$. 
Consequently, the restriction map of Coleman functions to each residue class
$]x[$  ($x\in{\mathbb P}^1(\overline{{\mathbb F}_p}))$
is injective.

(ii). {\bf Branch Independency Principle:}
For any branches $\varpi_1, \varpi_2\in\C_p$,
 there exists an isomorphism
$\iota_{\varpi_1,\varpi_2}:
\mathcal O^\loc_{\varpi_1}\simeq
\mathcal O^\loc_{\varpi_2}$
which acts as the identity in the first component of 
$\mathcal O_{\varpi_1}^\loc$
and 
replaces  $\log^{p,\varpi_1}$ with $\log^{p,\varpi_2}$ in its second component.
Similarly, an analogous isomorphism exists for differentials
 $\iota_{\varpi_1,\varpi_2}:\Omega^\loc_{\varpi_1}\simeq \Omega^\loc_{\varpi_2}$.
These mappings induce isomorphisms  
$\iota_{\varpi_1,\varpi_2}:\mathcal O^\col_{\varpi_1}\simeq \mathcal O^\col_{\varpi_2}$
and
$\iota_{\varpi_1,\varpi_2}:\Omega^\col_{\varpi_1}\simeq \Omega^\col_{\varpi_2}$.
The  Coleman integration maps are compatible with these isomorphisms,
as shown in the following  commutative diagram
\[
\begin{tikzcd}
\Omega^\col_{\varpi_1}\arrow[r, "\int_{(\varpi_1)}"] 
&  \mathcal O^\col_{\varpi_1}/\C_p \\
\Omega^\col_{\varpi_2} \arrow[u, "\iota_{\varpi_1,\varpi_2}"] \arrow[r, "\int_{(\varpi_2)}"]
& \mathcal O^\col_{\varpi_2}/\C_p  \arrow[u,, "\iota_{\varpi_1,\varpi_2}"'] .
\end{tikzcd}
 \]
\end{prop}

For further details on the theory of Coleman integration,
see such as \cite{Bes, Col}.

\begin{eg}[\cite{F04}]
\label{pMPL and pMZV}
(i).
For a tuple ${\mathbf k}=(k_1,\dots,k_n)\in\N^n$ ($n\in\N$),
the {\it $p$-adic multiple polylogarithm}, denoted by 
$\Li^{p,\varpi}_{\mathbf k}(z)\in\mathcal O^\col_\varpi$,
which  induces the map
$
\Li^{p,\varpi}_{\mathbf k}:{\mathcal X}(\C_p)\to \C_p
$ 
is introduced in \cite{F04}.
This function is constructed so as to satisfy 
the differential equation \eqref{eq: diff eq MPL}.
The special case $\Li_1^{p,\varpi}(z)=-\log^{p,\varpi}(1-z)$ recovers the $p$-adic logarithm.
Its restriction to $]0[$ is given by
the same power series expansion as in \eqref{eq:MPL}. 
On the tube $]\mathbb P^1\setminus\{0,1,\infty\}[$, 
it is independent of the choice of branch parameter $\varpi$; 
hence, from this point, the symbol $\varpi$ will occasionally be omitted.

(ii).
The {\it $p$-adic multiple zeta value}
 $\zeta_p({\bf k})\in\Q_p$ 
is defined as a certain limit of $\Li^{p,\varpi}_{\bf k}(z)$ as $z\to 1$.
This limit is taken under the condition that the ramification index of $z$
remains bounded.
This ensures convergence of the limit when $k_n>1$
and the independence of $\zeta_p(\bf k)$
from the choice of the branch parameter $\varpi$.
\end{eg}

\subsection{Formal version of $p$-adic hypergeometric function}\label{sec:def of pHG}
We introduce a formal version of $p$-adic hypergeometric function
by exploiting a fundamental solution of the $p$-adic KZ equation.

Let $\varpi\in\C_p$ be a chosen branch of the $p$-adic logarithm. 
In \cite{F04}, a $p$-adic solution
of the KZ equation \eqref{eq:KZ-equation} is constructed, 
along with its fundamental solution
$$G_{\vec{01}}^{p,\varpi}(\e_0,\e_1)(z) 
\in\mathcal O_\varpi^\col\langle\langle \e_0,\e_1\rangle\rangle
$$
which induces the map
$
G_{\vec{01}}^{p,\varpi}(\e_0,\e_1):{\mathcal X}(\C_p)\to
\C_p\langle\langle \e_0,\e_1\rangle\rangle.
$
This fundamental solution exhibits a specific asymptotic behavior
$z^{\e_0}$  in a neighborhood of $z=0$.
In \cite[Theorem 3.15]{F04}, it is further shown that  all coefficient of 
$G_{\vec{01}}^{p,\varpi}(\e_0,\e_1)(z)$ is described by combinations of
$p$-adic MPLs 
$\Li^{p,\varpi}_{\mathbf k}(z)$'s (cf. Example \ref{pMPL and pMZV}) and 
the $p$-adic logarithm $\log^{p,\varpi}(z)$.
These coefficients are described using exactly the same formulas as those 
in Proposition \ref{prop:explicit formulae}.
Similar to the construction in \cite{Dr89}, $p$-adic analogues of the other five fundamental solutions are given as follows:
\begin{align}\label{eq: pB}
G^{p,\varpi}_{\vec{10}}(\e_0,\e_1)(z)
&:=G^{p,\varpi}_{\vec{01}}(\e_1,\e_0)(1-z)
, \\ \notag
G^{p,\varpi}_{\vec{1\infty}}(\e_0,\e_1)(z)
&:=G^{p,\varpi}_{\vec{01}}(\e_1,\e_\infty)(1-\frac{1}{z})
, \\ \notag
G^{p,\varpi}_{\vec{\infty 1}}(\e_0,\e_1)(z)
&:=G^{p,\varpi}_{\vec{01}}(\e_\infty,\e_1)(\frac{1}{z})
, \\ \notag
G^{p,\varpi}_{\vec{\infty 0}}(\e_0,\e_1)(z)
&:=G^{p,\varpi}_{\vec{01}}(\e_\infty,\e_0)(\frac{1}{1-z})
, \\ \notag
G^{p,\varpi}_{\vec{0\infty}}(\e_0,\e_1)(z)
&:=G^{p,\varpi}_{\vec{01}}(\e_0,\e_\infty)(\frac{z}{z-1})
\end{align}
with $\e_\infty=-\e_0-\e_1$.
In \cite{F04}
the $p$-adic KZ associator $\varPhi^p_\KZ(\e_0,\e_1)\in\Q_p\langle\langle \e_0,\e_1\rangle\rangle$
is defined by the quotient
\begin{equation}\label{eq: Phi^p and G}
\varPhi^p_\KZ(\e_0,\e_1)=
G^{p,\varpi}_{\vec{10}}(\e_0,\e_1)(z)^{-1}\cdot
G^{p,\varpi}_{\vec{01}}(\e_0,\e_1)(z) .
\end{equation}
It is shown that this associator is independent of the choice of 
the branch parameter $\varpi$
and that each coefficient is expressed as a linear combination of
$p$-adic multiple zeta values $\zeta_p({\bf k})\in\Q_p$
(cf. Example \ref{pMPL and pMZV}).
We note that 
$G^{p,\varpi}_{\vec{01}}(\e_0,\e_1)(z)= G^{p,\varpi}_{\vec{0\infty}}(\e_0,\e_1)(z)$
due to $\zeta_p(2)=0$.


Let $\aaa,\bbb,\ccc$ be formal variables
and consider  $2\times 2$-matrices
$$
{\mathsf X}=
\begin{pmatrix}
0 & \bbb \\
0 & \uuu
\end{pmatrix} 
\qquad\text{and}\qquad
{\mathsf Y}=
\begin{pmatrix}
0 & 0 \\
\aaa & \vvv
\end{pmatrix}
$$
with
$\uuu=1-\ccc$ and $\vvv=\aaa+\bbb+1-\ccc$.
The substitution of
$(X,-Y)$ for $(\e_0,\e_1)$ in $G_{\vec{01}}^{p,\varpi}(\e_0,\e_1)(z)$
converges in
$\Mat_2({\mathcal O^\col_\varpi}[[\aaa,\bbb,\uuu]])$, 
and we denote the resulting matrix by
$
G^{p,\varpi}_{\vec{01}}({\mathsf X},-{\mathsf Y})(z).
$

Our definition of $\uHG{p,\varpi}{\aaa,\bbb}{\ccc}{z}$
is motivated by the equation \eqref{eq:HG=11}
in the complex case:

\begin{defn}\label{defn: formal version of pHG}
We define {\it the formal version of the $p$-adic hypergeometric function}
$\uHG{p,\varpi}{\aaa,\bbb}{\ccc}{z}$ 
for $z\in\X(\C_p)$
as its $(1,1)$-entry
\begin{equation*}\label{eq: HG in Ocol}
\uHG{p,\varpi}{\aaa,\bbb}{\ccc}{z}:=
[G^{p,\varpi}_{\vec{01}}({\mathsf X},-{\mathsf Y})(z)]_{(1,1)}
\in{\mathcal O^\col_\varpi}
[[\aaa,\bbb,\uuu]].
\end{equation*}
\end{defn}

This defines a map
$$
\uHG{p,\varpi}{\aaa,\bbb}{\ccc}{-}: {\mathcal X}(\C_p)\to\C_p[[\aaa,\bbb,\uuu]]
$$
which sends
$z\in{\mathcal X}(\C_p)\mapsto \uHG{p,\varpi}{\aaa,\bbb}{\ccc}{z}$.

By arguments analogous to those in the complex case, 
we can deduce a $p$-adic analogue of 
Ohno-Zagier formula \eqref{eq: expansion of HG}.

\begin{lem}\label{lem: p-adic Ohno-Zagier}
Let $\varpi\in\C_p$ be a chosen branch. Then the following identity holds in
 ${\mathcal O_\varpi^\col}[[\aaa,\bbb,\ccc-1]]$:
\begin{equation*}\label{eq: p-adic Ohno-Zagier}
\uHG{p,\varpi}{\aaa,\bbb}{\ccc}{z}=
1+ \aaa\bbb\sum_{\substack{k,n,s> 0 \\k\geqslant n+s, \ n \geqslant s}}
g^{p,\varpi}_0(k,n,s)(z)\uuu^{k-n-s}\vvv^{n-s}
(\aaa\bbb+\uuu\vvv)^{s-1} 
\end{equation*}
with $\uuu=1-\ccc$, $\vvv=\aaa+\bbb+1-\ccc=\aaa+\bbb+\uuu$ and
\begin{equation*}
g^{p,\varpi}_0(k,n,s)(z)=
\sum_{\substack{\wt({\mathbf k})=k, \ \mathrm{dp}({\mathbf k})=n, \ \mathrm{ht}({\mathbf k})=s \\ {\mathbf k}:\text{ admissible index}}} \Li^{p,\varpi}_{\mathbf k}(z)
\quad \in {\mathcal O^\col_\varpi}.
\end{equation*}
\end{lem}

\begin{proof}
It is shown in \cite[Theorem 3.15]{F04} that
the coefficients of $G_{\vec{01}}^{p}(\e_0,\e_1)(z)$ are expressed as polynomial combinations of $p$-adic multiple polylogarithms $\Li^{p,\varpi}_{\mathbf{k}}(z)$ and the $p$-adic logarithm $\log^{p,\varpi}(z)$, following precisely the same formulas presented in Proposition \ref{prop:explicit formulae} for the complex case. 

Leveraging this result and applying the matrix calculations detailed in \cite[Lemma 3.1 and (65)]{O},
we obtain the claimed formula.
\end{proof}



Our function also satisfies the hypergeometric differential equation:

\begin{prop}
The function 
$w=\uHG{p,\varpi}{\aaa,\bbb}{\ccc}{z}
=[G^{p,\varpi}_{\vec{01}}({\mathsf X},-{\mathsf Y})(z)]_{(1,1)}
\in{\mathcal O^\col_\varpi}[[\aaa,\bbb,\uuu]]
$
is a solution
of the differential equation
\eqref{eq:formal HG diff eq}.
Furthermore its derivative with respect to $z$ is given by
$\frac{dw}{dz}=\frac{\bbb}{z}[G^{p,\varpi}_{\vec{01}}({\mathsf X},-{\mathsf Y})(z)]_{(2,1)}$.
\end{prop} 

\begin{proof}
Applying the matrix calculations detailed in \cite[Lemma 3.1]{O},
we find that the $(2,1)$-entry of 
$[G^{p,\varpi}_{\vec{01}}({\mathsf X},-{\mathsf Y})(z)]_{(2,1)}$
is given by
\begin{equation*}
\aaa\sum_{\substack{k,n,s> 0 \\k\geqslant n+s, \ n \geqslant s}}
g^{p,\varpi}_1(k-1,n,s)(z)\uuu^{k-n-s}\vvv^{n-s}
(\aaa\bbb+\uuu\vvv)^{s-1} 
\in {\mathcal O_\varpi^\col}[[\aaa,\bbb,\uuu]]
\end{equation*}
with 
\begin{equation*}
g^{p,\varpi}_1(k,n,s)(z)=
\sum_{\substack{\wt({\mathbf k})=k, \ \mathrm{dp}({\mathbf k})=n, \ \mathrm{ht}({\mathbf k})=s }} \Li^{p,\varpi}_{\mathbf k}(z)
\quad \in {\mathcal O^\col_\varpi}.
\end{equation*}
This expression coincides with
$\frac{z}{\bbb}\frac{d}{dz}\left(\uHG{p,\varpi}{\aaa,\bbb}{\ccc}{z}\right)$.
Consequently, the first column of
$G^{p,\varpi}_{\vec{01}}({\mathsf X},-{\mathsf Y})(z)$
provides a $p$-adic solution of the system:
\begin{equation*}
\frac{d}{dz}\vec{g}
=\left\{\frac{1}{z}{\mathsf X}+\frac{1}{1-z}{\mathsf Y}\right\}\cdot \vec{g}
\end{equation*}
with
$\vec{g}=\vec{g}(w)=
\begin{pmatrix}
w \\ \frac{z}{\bbb}\frac{dw}{dz}
\end{pmatrix}\in {\mathcal O_\varpi^\col}[[\aaa,\bbb,\uuu]]^{\oplus 2}$.
Analogous to the complex case, this differential equation can be reformulated to yield \eqref{eq:formal HG diff eq}.
\end{proof} 

\begin{rem}
As shown in \cite[\S 3]{O} in the complex case,
the $(2,1)$-entry of $[G^{p,\varpi}_{\vec{01}}({\mathsf X},-{\mathsf Y})(z)]$ is given by
$\langle z\rangle^\uuu \cdot\uHG{p,\varpi}{\aaa+\uuu,\bbb+\uuu}{1+\uuu}{z}$
where $\langle z\rangle^\uuu:=\sum_{n=0}^\infty\frac{(\log^{p,\varpi}(z)\cdot \uuu)^n}{n!}$.
\end{rem}

\subsection{Local behavior around $z=0$}\label{subsec: local behavior at 0}
This subsection proves the claim of 
Theorems \ref{thm:main_result} and \ref{thm:extended_convergence}
in the case of $|z|_p<1$.



Since the restriction of $\Li^{p,\varpi}_{\bf k}(z)$ to $]0[$ is given by 
\eqref{eq:MPL} in $\Q[[z]]$,
the restriction of the formal version 
$\uHG{p,\varpi}{\aaa,\bbb}{\ccc}{z}\in{\mathcal O_\varpi^\col}[[\aaa,\bbb,\uuu]]$
to $]0[$
 determines the series 
\begin{equation*}
\uHG{p,\varpi}{\aaa,\bbb}{\ccc}{\zzz}_{]0[}\in\Q[[\aaa,\bbb,\uuu,\zzz]]
\end{equation*}
with $\uuu=1-\ccc$ by Lemma \ref{lem: p-adic Ohno-Zagier}.

\begin{prop}
The restriction $\uHG{p,\varpi}{\aaa,\bbb}{\ccc}{\zzz}_{]0[}$ 
is independent of  any choice of branch parameter $\varpi$
(from this point onward, we suppress the symbol $\varpi$)
and agrees with the usual power series expansion, that is, 
we have
\begin{equation}\label{eq: formal pHG at 0}
\uHG{p}{\aaa,\bbb}{\ccc}{\zzz}_{]0[}
=\sum_{n=0}^\infty\frac{(\aaa)_n(\bbb)_n}{(\ccc)_n n!}\zzz^n
\in
\Q[\aaa,\bbb,\uuu]_{(\uuu)} [[\zzz]],
\end{equation}
where $\Q[\aaa,\bbb,\uuu]_{(\uuu)}$ means the localization of
$\Q[\aaa,\bbb,\uuu]$ at the prime ideal $(\uuu)$.
\end{prop}

\begin{proof}
The formulas \eqref{def eq: HG}
and \eqref{eq: expansion of HG} in the complex case follow the identity:
$$
\sum_{n=0}^\infty\frac{(a)_n(b)_n}{(c)_n n!}z^n
=1+ ab\sum_{\substack{k,n,s> 0 \\k\geqslant n+s, \\ n \geqslant s}}
\sum_{\substack{\wt({\mathbf k})=k, \\ \mathrm{dp}({\mathbf k})=n, \\ \mathrm{ht}({\mathbf k})=s }} \Li^{\Q}_{\mathbf k}(z)
u^{k-n-s}v^{n-s}
(ab+uv)^{s-1},
$$
which holds for variables
$a$, $b$, $u$, $z$ in a convergent domain.
This identity lifts to an equation with the same expression in the rational structure
$\Q[[\aaa,\bbb,\uuu,\zzz]]$:
$$
\sum_{n=0}^\infty\frac{(\aaa)_n(\bbb)_n}{(\ccc)_n n!}\zzz^n
=1+ \aaa\bbb\sum_{\substack{k,n,s> 0 \\k\geqslant n+s, \\ n \geqslant s}}
\sum_{\substack{\wt({\mathbf k})=k, \\ \mathrm{dp}({\mathbf k})=n, \\ \mathrm{ht}({\mathbf k})=s }} \Li^{\Q}_{\mathbf k}(\zzz)
\uuu^{k-n-s}\vvv^{n-s}
(\aaa\bbb+\uuu\vvv)^{s-1}.
$$
Since the restriction
$\uHG{p,\varpi}{\aaa,\bbb}{\ccc}{\zzz}_{]0[}$ is given by 
the right hand side of the above equation and 
$\frac{1}{(\ccc)_n}=\frac{1}{(1-\uuu)(2-\uuu)\cdots (n-\uuu)}$ is in $\Q[\uuu]_{(\uuu)}$
with $\uuu=1-\ccc$,
we see that
$\uHG{p,\varpi}{\aaa,\bbb}{\ccc}{\zzz}_{]0[}$ is given by
\eqref{eq: formal pHG at 0} in 
$\Q[\aaa,\bbb,\uuu]_{(\uuu)}[[\zzz]]$.
\end{proof}

%

By \eqref{eq: formal pHG at 0}, we
regard $\uHG{p}{\aaa,\bbb}{\ccc}{\zzz}_{]0[}\in 
\Q[\aaa,\bbb,\uuu]_{(\uuu)} [[\zzz]]$,
which allows
the specialization $(\aaa,\bbb,\ccc)\mapsto(\alpha,\beta,\gamma)\in\Q_p^3$
to be carried out whenever the parameters avoid the poles of the expansion--namely, when
they are not  located on poles
(that is, when $\gamma\neq 0,-1,-2,\dots$).
In such cases, the specialized function is given by
\begin{equation}\label{eq: pHG specialization abc}
\uHG{p}{\alpha,\beta}{\gamma}{\zzz}_{]0[}
=\sum_{n=0}^\infty\frac{(\alpha)_n(\beta)_n}{(\gamma)_n n!}\zzz^n
\in\Q_p[[\zzz]].
\end{equation}

To investigate its convergence,
we prepare a fundamental (and maybe well-known) lemma below.

\begin{lem}\label{lem:wellknown}
Let $n$ be a positive integer whose $p$-adic expansion is given by 
$n=a_0+a_1p+\cdots+a_rp^r$  ($a_i\in [0,p-1]$).
Put $s_p(n)=a_0+a_1+\cdots+a_r$.
Then we have
$$v_p(n!)=\frac{n-s_p(n)}{p-1}.
$$
Here, 
$v_p(x)$ is the additive $p$-adic valuation.
\end{lem}

\begin{proof}
By Legendre's formula, the $p$-adic valuation of $n!$ is given as
$
v_p(n!) = \sum_{s=1}^\infty \left\lfloor \frac{n}{p^s} \right\rfloor,
$
which is calculated to be $\frac{n-(a_0+a_1+\cdots+a_r)}{p-1}$.
\end{proof}

\begin{thm}\label{thm: convergence of pHG on 0}
(1).
Put 
$\alpha,\beta,\gamma\in {\Z_p}$.
Let
$\sum_{i=0}^\infty c^-_ip^{m_i}$
with 
$0\leqslant m_0<m_1<\cdots$ and $c^-_i\in [1,p-1]$ (we note that it says
$c^-_i\neq 0$ for all $i$)
be the $p$-adic expansion of $-\gamma\in\Z_p$.
Assume that 
\begin{equation}\label{eq:assumption}
(S:=)\ \varlimsup_n \frac{m_n}{p^{m_{n-1}}}  < \infty.
\end{equation}
Then 
the specialization \eqref{eq: pHG specialization abc}
converges 
on $\{\zzz\in\C_p\bigm| |\zzz|_p<p^{-\frac{S}{p-1}}\}$.

(2).
Let 
$\alpha,\beta\in\Z_p$ and $\gamma\in {\Z_{(p)}}:=\Z_p\cap \Q$ 
with  $\gamma \neq 0,-1,-2,\dots$.
Then  
the series 
\eqref{eq: pHG specialization abc}
converges on $\{\zzz\in\C_p\bigm||\zzz|_p<1\}$.
\end{thm}

\begin{proof}
(1).
Let $R$ be  the radius of convergence  of 
$\uHG{p}{\alpha,\beta}{\gamma}{\zzz}\in\Q_p[[\zzz]]$
and denote $d_n$ to be the coefficient  of $\zzz^n$  in $\uHG{p}{\alpha,\beta}{\gamma}{\zzz}$;
$$
d_n=\frac{(\alpha)_n}{n!}\cdot\frac{(\beta)_n}{n!}\cdot\frac{n!}{(\gamma)_n}.
$$
By 
$\frac{(\alpha)_n}{n!}, \frac{(\beta)_n}{n!}\in\Z_p$,
we have
$$
\frac{1}{R}=\varlimsup_n |d_n|_p^{\frac{1}{n}}
\leq
\varlimsup_n \left|\frac{n!}{(\gamma)_n}\right|_p^{\frac{1}{n}}.
$$
Therefore
$$
-\log_pR
\leq
\varlimsup_n \frac{-1}{n}v_p(\frac{n!}{(\gamma)_n})
=\varlimsup_n \frac{-1}{n}\left\{
\frac{n-s_p(n)}{p-1}-
v_p\left({(\gamma)_n}\right)
\right\}.
$$
For $\eta\in\Z_p$ with the $p$-adic digit expansion $\eta=e_0+e_1p+e_2p^2+\cdots$
($e_i\in[0,p-1]$)
and $k\geq 0$,
we put 
$$t(\eta;k)=e_0+e_1p+\cdots+e_{k}p^{k}$$
and 
$$T(\eta;k):=\min\{t(\eta;h)\bigm| t(\eta;h)\neq t(\eta;k), h>k \},$$
which is well-defined for any $k\geqslant 0$
unless $\eta$ is a positive integer or $0$.
Put $l_n=t(-\gamma;\lfloor\log_pn\rfloor)$,
$L_n=T(-\gamma;\lfloor\log_pn\rfloor)$ and $m_n=\lfloor \log_pL_n\rfloor$,
where  $\log_p$ denotes the (real) logarithm to base $p$, not the 
$p$-adic logarithm introduced in \S \ref{subsec: review on Coleman integrations}.
Then we have $n<L_n$,  
$p^{m_n}\leq L_n<p^{m_n+1}$  and $-\gamma\equiv L_n\not\equiv 0 \bmod p^{m_n+1}$.
So we have
\begin{align*}
v_p((\gamma)_n)
&=v_p\left(\gamma(\gamma+1)\cdots(\gamma+n-1)\right)
=v_p\left((-\gamma)(-\gamma-1)\cdots(-\gamma-(n-1))\right) \\
&=v_p\left(\frac{L_n!}{(L_n-n)!}\right)
=\frac{1}{p-1}\{n-s_p(L_n)+s_p(L_n-n)\}
\end{align*}
by Lemma \ref{lem:wellknown}.
Therefore
$$
-\log_pR
\leq
\varlimsup_n\frac{1}{n(p-1)}\{s_p(n)-s_p(L_n)+s_p(L_n-n)\}. 
$$
Since both $s_p(n)$ and $s_p(L_n)$ are bounded by the order of $\log_p(n)$,
we obtain
$$
-\log_pR
\leq
\frac{1}{p-1}\varlimsup_n\frac{s_p(L_n-n)}{n}. 
$$

Since by definition we have
$$
l_{p^{m_{i-1}}-1}<l_{p^{m_{i-1}}}=l_{p^{m_{i-1}}+1}=\cdots=l_{p^{m_i}-1}
<l_{p^{m_i}},
$$
we have
$$
L_{p^{m_{i-1}}-1}<L_{p^{m_{i-1}}}=L_{p^{m_{i-1}}+1}=\cdots=L_{p^{m_i}-1}
=l_{p^{m_i}}
<L_{p^{m_i}}.
$$

Therefore
$$
\frac{1}{p-1}
\varlimsup_n\frac{s_p(L_n-n)}{n}
\leq
\frac{1}{p-1}
\varlimsup_i\frac{(p-1)m_i}{p^{m_{i-1}}}
=\frac{S}{p-1},
$$
we have
$
-\log_pR\leq \frac{S}{p-1},
$
which says $R\geq p^{-\frac{S}{p-1}}$. 

(2).
It is because $S$ in \eqref{eq:assumption} is $0$ for such $\gamma$.
\end{proof}

Thus, Theorems \ref{thm:main_result} and \ref{thm:extended_convergence} are established in the case $|z|_p<1$. 
\qed

%



\subsection{Local behavior on $]\mathbb P^1\setminus\{0,1,\infty\}[$}
\label{subsec: prolongation on the rest disks}
An analytic continuation of the 
$p$-adic hypergeometric function is performed, residue-wise, 
to the domain $]\mathbb P^1\setminus\{0,1,\infty\}[$,
utilizing the bounding techniques of Chatzistamatiou~\cite{Ch},
in Theorem \ref{thm: convergence of pHG on P-01infty}.
This allows us to establish Theorems ~\ref{thm:main_result} and ~\ref{thm:extended_convergence} 
in this setting.

Let $\xi\in \C_p$ be any root of unity whose order is prime to $p$
and assume $\xi\neq 1$.
Since the restrictions of all $\Li_{\bf k}(z)\in{\mathcal O^\col_\varpi}$ to the residue class $]\bar\xi[$ lie within
$\C_p[[z-\xi]]$,
the restriction of $\uHG{p,\varpi}{\aaa,\bbb}{\ccc}{z}\in {\mathcal O^\col_\varpi}[[\aaa,\bbb,\uuu]]$
to $]\bar\xi[$ 
is expressed as the formal power series
\begin{equation}\label{eq: formal HG at xi}
\uHG{p,\varpi}{\aaa,\bbb}{\ccc}{\zzz}_{]\bar\xi[}\in \C_p[[\aaa,\bbb,\uuu,\zzz-\xi]]
\end{equation}
with $\uuu=1-\ccc$ by Lemma \ref{lem: p-adic Ohno-Zagier}.

\begin{thm}\label{thm: convergence of pHG on P-01infty}
(1).
The restriction $\uHG{p,\varpi}{\aaa,\bbb}{\ccc}{\zzz}_{]\bar \xi[}$ is independent of any choice of branch parameter $\varpi$
(from this point onward, we suppress the symbol $\varpi$)
and
satisfies
$$\uHG{p}{\aaa,\bbb}{\ccc}{\zzz}_{]\bar\xi[}\in
\Q_p(\xi)[[\aaa,\bbb,\uuu,\zzz-\xi]].
$$

(2).
Let $z\in\C_p$ and  $\alpha,\beta,\gamma\in\Z_p$.
Then, upon the substitution
$(\aaa,\bbb,\ccc,\zzz)\mapsto (\alpha,\beta,\gamma,z)$
in \eqref{eq: formal HG at xi},
the resulting series converges whenever  $|z-\xi|_p<p^{-\frac{1}{p-1}}$.

\end{thm}

\begin{proof}
%
(1).
For an integer $k\geqslant 0$, 
we consider the PD-ideal ${\mathfrak p}^{[k]}$ of 
$\Z_p[\xi]$  defined by 
$$
{\mathfrak p}^{[k]}:=\langle \frac{p^m}{m!}\bigm| m\geq k\rangle_{\Z_p[\xi]}.
$$
We put ${\mathfrak p}^{[k]}={\mathfrak p}^{[0]}$ for an integer $k<0$.
By \cite[Theorem 5.3]{Ch}, it holds that
$$\Li^p_{\mathbf k}(\xi)\in {\mathfrak p}^{[\wt(\mathbf k)]}.$$
Then, recursively applying the differential equation
$$
\frac{d}{d\zzz}\Li^p_{k_1,\dots,k_m}(\zzz)
=\begin{cases}
\frac{1}{\zzz}\Li^p_{k_1,\dots,k_m-1}(\zzz) & (k_m>1), \\
\frac{1}{1-\zzz}\Li^p_{k_1,\dots,k_{m-1}}(\zzz) & (k_m=1),
\end{cases}
$$
we see that 
$$
\Li^{p,(n)}_{\mathbf k}(\xi)\in {\mathfrak p}^{[\wt(\mathbf k)-n]}
$$
since we have $v_p(\xi)=v_p(1-\xi)=0$ for roots of unity
whose order  is prime to $p$ and $\xi\neq 1$.
Here $\Li_{\mathbf k}^{p,(n)}(\zzz)$ means the $n$-th derivative of $\Li^p_{\mathbf k}(\zzz)$ with respect to $\zzz$.

By the Taylor expansion
$$
\Li^p_{\mathbf k}(\zzz)_{]\bar\xi[}=\sum_{n=0}^\infty \Li^{p,(n)}_{\mathbf k}(\xi)\frac{(\zzz-\xi)^n}{n!},
$$
and Lemma \ref{lem: p-adic Ohno-Zagier},
we obtain 
\begin{equation}\label{eq: HG expansion around xi}
\uHG{p}{\aaa,\bbb}{\ccc}{\zzz}_{]\bar\xi[}=
\sum_{i=0}^\infty \ccc_{\xi,i}\frac{(\zzz-\xi)^i}{i!}
\end{equation}
with
\begin{align*}\label{eq: c xi i}
&\ccc_{\xi,i}:=\delta_{i,0}+
\aaa\bbb\sum_{\substack{k,n,s>0 \\k \geqslant n+s, \ n \geqslant s}}
g^{(i)}_{0,\xi}(k,n,s)\uuu^{k-n-s}\vvv^{n-s}
(\aaa\bbb+\uuu\vvv)^{s-1}, \\ \notag
&g^{(i)}_{0,\xi}(k,n,s):=\sum_{\substack{\wt({\mathbf k})=k, \ \mathrm{dp}({\mathbf k})=n, \ \mathrm{ht}({\mathbf k})=s \\ {\mathbf k}:\text{ admissible index}}} \Li^{p,(i)}_{\mathbf k}(\xi)
\in {\mathfrak p}^{[k-i]}.
\end{align*}
Its coefficient is in $\Q_p(\xi)$
which is independent of any choice of branch parameter $\varpi$.

(2).
Let $c_{\xi,i}$ denote the value obtained by specializing at $(\aaa,\bbb,\ccc)=(\alpha,\beta,\gamma)$
with $\alpha,\beta, \gamma\in {\Z_p}$.
This value converges in $\mathcal O_{\C_p}$,
because  $g^{(i)}_{0,\xi}(k,n,s)\in {\mathfrak p}^{[k-i]}$.
By Lemma \ref{lem:wellknown} and $c_{\xi,i}\in\Z_p[\xi]$,
we obtain
\begin{align*}
v_p\left(\frac{(z-\xi)^nc_{\xi,n}}{n!}\right)
&=n\cdot v_p(z-\xi)+v_p(c_{\xi,n})-\frac{n-s_p(n)}{p-1}\\
&\geqslant n\left(v_p(z-\xi)-\frac{1}{p-1}\right)+\frac{s_p(n)}{p-1}.
\end{align*}
When $|z-\xi|_p<p^{-\frac{1}{p-1}}$, this tends to $\infty$ when $n\mapsto\infty$.
Therefore  the series \eqref{eq: HG expansion around xi} converges for
$|z-\xi|_p<p^{-\frac{1}{p-1}}$.
\end{proof}

\smallskip
Thus, Theorems \ref{thm:main_result} and \ref{thm:extended_convergence} are established in the case when 
$\bar z\neq \bar 0,\bar 1,\bar\infty$.
Note that in this situation, one may select 
$\xi\in\mu_{p-1}\subset \Q_p$.
\qed
\smallskip
\begin{rem}
The specialization
$\uHG{p}{\alpha,\beta}{\gamma}{z}_{]\bar\xi[}\in \C_p[[z-\xi]]$
is a solution of the $p$-adic hypergeometric differential equation
\begin{equation*}\label{eq: HG difff eq}
z(1-z)\frac{d^2w}{dz^2}+\{\gamma-(\alpha+\beta+1)z\}\frac{dw}{dz}-\alpha\beta w=0.
\end{equation*}
According to \cite[Example 9.6.2]{Ke},
a general series solution of the hypergeometric
differential equation at $z=z_0\in\mathcal O_{\C_p}$,
that is not congruent to $0$ and $1$,
converges on $|z-z_0|<1$
when $\alpha, \beta, \gamma\in\Z_{(p)}$.
It therefore indicates that  
$\uHG{p}{\alpha,\beta}{\gamma}{z}_{]\bar\xi[}$  is a rigid-analytic function on 
$]\bar\xi[$.
\end{rem}

\subsection{Local behavior around  $z=\infty$}\label{subsec: local behavior at infty}
We investigate the convergence of the 
$p$-adic hypergeometric function in the residue class of $\infty$
in Theorem \ref{thm: prolongation to infty}.
Our approach is to deduce convergence properties at $\infty$
by establishing a connection formula 
in Theorem \ref{thm: formal connection formula between 0 and infinity}
that relates the points $0$ and $\infty$.
These arguments are used to justify the validity of 
Theorems ~\ref{thm:main_result} in this setting.

The following lemma is essential for establishing convergence in our setting.
\begin{lem}\label{lem: power map convergence}
Let $K$ be a field extension  of $\Q_p$ in $\C_p$ with 
$$e_K<p$$
where $e_K$ means  the absolute ramification index of $K$, i.e.
$v_p(K^\times)=\frac{1}{e_K}\Z$.
Assume that the branch  $\varpi\in\C_p$ of the $p$-adic logarithm 
(whence we have $\varpi=\log^{p,\varpi}(p)$)
and  $\lambda\in\C_p$ 
are chosen  to satisfy
 $$|\lambda|_p<1 \text{  and   }|\lambda\varpi|_p< p^{-\frac{1}{p-1}}.$$
Let $z\in K$.
 Then  
 \begin{equation}\label{eq: lambda power of z}
 \langle z\rangle^\lambda:=\exp\{\lambda\log^{p,\varpi} (z) \}=\sum_{n=0}^\infty\frac{(\lambda\log^{p,\varpi} (z))^n}{n!}
 \end{equation}
 converges in $\C_p$
 unless $z=0$.
 \end{lem}

It should be noted that $\langle z\rangle^\lambda$  depends on the choice of the branch  $\varpi$
of the $p$-adic logarithm when $z$ lies in the analytic neighborhoods $]0[$ or $]\infty[$.
 
\begin{proof}
When $z\neq 0\in K$ is written as $\epsilon p^r(1+m)$ with 
$\epsilon\in\mu_\infty$,
$r\in\Q$ and $m\in\mathfrak m_{K}$,
we have $|\log^{p,\varpi}(1+m)|_p=|m|_p$ because $|m|_p\leqslant p^{-\frac{1}{p-1}}$ by $z\in K$
(cf. \cite[Lemma 5.5]{W}).
By
$$
\lambda\log^{p,\varpi} (z)=r\lambda\log^{p,\varpi}(p)+\lambda\log^{p,\varpi}(1+m)
=r\lambda\varpi+\lambda\log^{p,\varpi}(1+m),
$$ 
$r\in\frac{1}{e_K}\Z\subset\Z_{(p)}$ and
 our assumption, we have 
 $$|\lambda\log^{p,\varpi} (z) |_p < p^{-\frac{1}{p-1}}.
 $$
Therefore
$
\langle z\rangle^\lambda =\exp\{\lambda\log^{p,\varpi} (z) \}
$
converges.
\end{proof}

The restrictions of $\Li^{p,\varpi}_{\bf k}(z)$ and $\log^{p,\varpi}(z)\in {\mathcal O^\col_\varpi}$ to $]\infty[$ are in
$A_\varpi^{\log}(]\infty[)$.
By the inclusion 
$$
A_\varpi^{\log}(]\infty[)\subset
\C_p[[\frac{1}{z}]][\log^{p,\varpi}(z)]
\subset
\C_p[\log^{p,\varpi}(z)][[\frac{1}{z}]]
$$
it follows from  Lemma \ref{lem: p-adic Ohno-Zagier}
that the restriction of 
$\uHG{p,\varpi}{\aaa,\bbb}{\ccc}{z}\in {\mathcal O^\col_\varpi}[[\aaa,\bbb,\uuu]]$
to $]\infty[$ 
is expressed as a formal power series
\begin{equation}\label{eq: formal HG at infinity}
\uHG{p, \varpi}{\aaa,\bbb}{\ccc}{\zzz}_{]\infty[}\in
\C_p[\LL_\infty][[\aaa,\bbb,\uuu,\frac{1}{\zzz}]],
\end{equation}
where $\LL_\infty$ stands for $\log^{p,\varpi} (\frac{1}{\zzz})$.

\begin{thm}\label{thm: formal connection formula between 0 and infinity}
The following formula holds:
\begin{equation}\label{eq: formal connection formula between 0 and infinity}
\uHG{p, \varpi}{\aaa,\bbb}{\ccc}{\zzz}_{]\infty[}
={\mathsf h}_1 \cdot {\langle \zzz\rangle}^{-\aaa}\cdot\uHG{p, \varpi}{\aaa,\aaa+1-\ccc}{\aaa-\bbb+1}{\frac{1}{\zzz}}_{]0[}+
{\mathsf h}_2 \cdot {\langle \zzz\rangle}^{-\bbb}\cdot\uHG{p, \varpi}{\bbb,\bbb+1-\ccc}{\bbb-\aaa+1}{\frac{1}{\zzz}}_{]0[}
\end{equation}
in $\Q_p[\LL_\infty][[\aaa,\bbb,\uuu,\frac{1}{\zzz}]]$
where
$$
\langle \zzz\rangle^{\mathsf{x}}:=\exp\{\mathsf{x}\LL_\infty \}=\sum_{n=0}^\infty\frac{(\mathsf{x}\LL_\infty)^n}{n!}
$$
for $\mathsf{x}=-\aaa, -\bbb$.
The coefficients 
${\mathsf h}_1$  and ${\mathsf h}_2$ 
are given 
as the $(1,1)$- and $(2,1)$-entries, respectively,  of the matrix
\begin{equation*}\label{eq: A and B}
{\mathsf H}:=
\begin{pmatrix}
1 & 1 \\
-\frac{\aaa}{\bbb} & -1
\end{pmatrix}^{-1}\cdot
\Phi_{\KZ}^p(X,Y-X)\cdot
\begin{pmatrix}
1 & 1 \\
0 & \frac{\uuu}{\bbb} 
\end{pmatrix}
\in \GL_2\left(\Q_p[[\aaa,\bbb,\uuu]]\left[\frac{1}{(\aaa-\bbb)\bbb}\right]\right), 
\end{equation*}
and, in fact,  ${\mathsf h}_1, {\mathsf h}_2\in\Q_p[[\aaa,\bbb,\uuu]]$.
\end{thm}

\begin{proof}
Set
$$\PP={\mathcal O}^\col_\varpi[[\aaa,\bbb,\uuu]].$$
We consider the following matrices, paralleling the constructions in \eqref{eqC} for the complex case:
\begin{align*}\label{eq p C}
&\V^{p,\varpi}_{\vec{01}}(z):=
G^{p,\varpi}_{\vec{01}}(X,-Y)(z)\cdot
\begin{pmatrix}
1 & 1 \\
0 & \frac{\uuu}{\bbb}
\end{pmatrix}
\in \GL_2(\PP[\frac{1}{\bbb}])
, \\ 
&\V^{p,\varpi}_{\vec{\infty 1}}(z):=
G^{p,\varpi}_{\vec{\infty 1}}(X,-Y)(z)\cdot
\begin{pmatrix}
1 & \Add{1} \\
\frac{-\aaa}{\bbb} & -1
\end{pmatrix}
\in \GL_2(\PP[\frac{1}{\bbb}]).
\end{align*}
Here
$\uuu=1-\ccc$,  $\vvv=\aaa+\bbb+1-\ccc=\aaa+\bbb+\uuu$ and
$
X=
\begin{pmatrix}
0 & \bbb \\
0 & \uuu
\end{pmatrix}, \
Y=
\begin{pmatrix}
0 & 0 \\
\aaa & \vvv
\end{pmatrix}\in \Mat_2(\Q[\aaa,\bbb,\ccc-1])
$
as in \S \ref{sec:def of pHG}.

By \eqref{eq: pB} and \eqref{eq: Phi^p and G}, we have 
$$
G^{p,\varpi}_{\vec{01}}(\e_0,\e_1)(z)=G^{p,\varpi}_{\vec{\infty 1}}(\e_0,\e_1)(z)\Phi_{\KZ}^p(\e_0,\e_\infty),
$$
$$
G^{p,\varpi}_{\vec{\infty 1}}(\e_0,\e_1)(z)=G^{p,\varpi}_{\vec{01}}(\e_\infty,\e_1)(\frac{1}{z})
$$ 
with $\e_\infty=-\e_0-\e_1$.
Therefore we obtain the transformation:
\begin{equation}\label{eq:Vp 01 and infty 1}
\V^{p,\varpi}_{\vec{01}}(z)=\V^{p,\varpi}_{\vec{\infty 1}}(z)\cdot
{\mathsf H}.
\end{equation}
\begin{equation*}\label{eq: V infty1 G 01} 
\V^{p,\varpi}_{\vec{\infty 1}}(z)=G^{p,\varpi}_{\vec{01}}(Y-X,-Y)(\frac{1}{z})
\cdot
\begin{pmatrix}
1 & \Add{1} \\
\frac{-\aaa}{\bbb} & -1
\end{pmatrix}.
\end{equation*}
Since $\V^{p,\varpi}_{\vec{01}}(z)$ and $\V^{p,\varpi}_{\vec{\infty 1}}(z)$ are
in $\GL_2(\PP[\frac{1}{\bbb}])$,
it follows that ${\mathsf H}$ is also contained in this group.
Furthermore, as ${\mathsf H}$ is
independent of $z$, we have  ${\mathsf H}\in \GL_2(\C_p[[\aaa,\bbb,\uuu]][\frac{1}{\bbb}])$.
Therefore we have 
\begin{equation*}
{\mathsf h}_1, {\mathsf h}_2\in \C_p[[\aaa,\bbb,\uuu]][\frac{1}{\bbb}].
\end{equation*}

In the quotient field of $\C_p[[\aaa,\bbb,\uuu]][\frac{1}{\bbb}]$, 
it is calculated to be
\begin{equation}\label{eq: def eq for h1 and h2}
\begin{pmatrix} {\mathsf h}_1 \\ {\mathsf h}_2 \end{pmatrix}:=
\frac{1}{\aaa-\bbb}
\begin{pmatrix}
-\bbb & -\bbb \\
\aaa &  \bbb 
\end{pmatrix}\cdot
\Phi_{\KZ}^p(X,Y-X)\cdot
\begin{pmatrix}
1 \\ 0 \\
\end{pmatrix}.
\end{equation}
Since we have $\Phi_{\KZ}^p(X,Y-X)\in\GL_2(\Q_p[[\aaa,\bbb,\uuu]])$,
it follows that
\begin{equation*}
{\mathsf h}_1, {\mathsf h}_2\in \frac{1}{\aaa-\bbb}\Q_p[[\aaa,\bbb,\uuu]].
\end{equation*}

Since 
$\C_p[[\aaa-\bbb,\bbb,\uuu]][\frac{1}{\bbb}]\cap 
\Q_p[[\aaa-\bbb,\bbb,\uuu]][\frac{1}{\aaa-\bbb}]=
\Q_p[[\aaa-\bbb,\bbb,\uuu]]=\Q_p[[\aaa,\bbb,\uuu]]
$,
we  have
\begin{equation*}
{\mathsf h}_1, {\mathsf h}_2\in \Q_p[[\aaa,\bbb,\uuu]].
\end{equation*}

Comparison of the $(1,1)$-entries on both sides of~\eqref{eq:Vp 01 and infty 1} yields
\begin{equation}\label{eq:uHG A B V infty 1}
\uHG{p,\varpi}{\aaa,\bbb}{\ccc}{z}
= [\V^{p,\varpi}_{\vec{\infty 1}}(z)]_{(1,1)} \cdot {\mathsf h}_1+
[\V^{p,\varpi}_{\vec{\infty 1}}(z)]_{(1,2)} \cdot {\mathsf h}_2
\end{equation}
in $\PP[\frac{1}{\bbb}]$.
By restricting \eqref{eq:uHG A B V infty 1}
to $]\infty[$ and combining
it with Lemma \ref{lem: formal connection formula between 0 and infinity} below, we derive
the connection formula
\eqref{eq: formal connection formula between 0 and infinity}.
\end{proof}


\begin{lem}\label{lem: formal connection formula between 0 and infinity}
The (1,1)- and (1,2)-entries of the restriction of the matrix 
$\V^{p,\varpi}_{\vec{\infty 1}}(z)$ to to the tubular neighborhood $]\infty[$
are given by 
\begin{equation}\label{eq:VV infty= FF0}
\begin{pmatrix} 
{\langle \zzz\rangle}^{-\aaa}\uHG{p, \varpi}{\aaa,\aaa+1-\ccc}{\aaa-\bbb+1}{\frac{1}{z}}_{]0[}, &
{\langle \zzz\rangle}^{-\aaa}\uHG{p, \varpi}{\bbb,\bbb+1-\ccc}{\bbb-\aaa+1}{\frac{1}{z}}_{]0[}
\end{pmatrix}.
\end{equation}
\end{lem}

\begin{proof}
In the complex case, by \eqref{eqB} and \eqref{eqC}, we have
$$
\V^\C_{\vec{\infty 1}}(z)
=G^\C_{\vec{01}}(Y_0-X_0,-Y_0)(\frac{1}{z})\cdot
\begin{pmatrix}
1 & \Add{1} \\
\frac{-a}{b} & -1
\end{pmatrix}.
$$
From \eqref{eqD}, its first row of this matrix
is computed as
$$
\begin{pmatrix}
[\V^\C_{\vec{\infty 1}}(z)]_{(1,1)}, & [\V^\C_{\vec{\infty 1}}(z)]_{(1,2)}
\end{pmatrix}
=
\begin{pmatrix} 
z^{-a}\uHG{\C}{a,a+1-c}{a-b+1}{\frac{1}{z}}, &
z^{-b}\uHG{\C}{b,b+1-c}{b-a+1}{\frac{1}{z}}
\end{pmatrix}
$$
for $(a,b,c)\in\C^3$  under appropriate analytic conditions.
Upon interpreting $a$, $b$, $c-1$ and $\frac{1}{z}$ as formal variables
—specifically, by considering their formal expansions near the origin—
and employing the formulas in Proposition \ref{prop:explicit formulae},
we see that
$[\V^\C_{\vec{\infty 1}}(z)]_{(1,1)}$ naturally defines an element,
denoted by 
$$[\V^\Q_{\vec{\infty 1}}(\zzz)]_{(1,1)} 
\in \frac{1}{\bbb}\Q[\LL_\infty][[\aaa,\bbb,\uuu,\frac{1}{\zzz}]],
$$
where $\LL_\infty$ corresponds to $\log (\frac{1}{\zzz})$.
Similarly by the expansion \eqref{def eq: HG},
$z^{-a}\uHG{\C}{a,a+1-c}{a-b+1}{\frac{1}{z}}$  determines an element,
denoted by 
$$
{\langle \zzz\rangle}^{-\aaa}\uHG{\Q}{\aaa,\aaa+1-\ccc}{\aaa-\bbb+1}{\frac{1}{\zzz}}\in
\Q[\LL_\infty][\aaa,\bbb,\uuu]_{(\aaa-\bbb)} [[\frac{1}{\zzz}]]\cap \Q[\LL_\infty][[\aaa,\bbb,\uuu,\frac{1}{\zzz}]].
$$
Analogously we have elements
$[\V^\Q_{\vec{\infty 1}}(\zzz)]_{(1,2)}$
in $\Q[\LL_\infty][[\aaa,\bbb,\uuu,\frac{1}{\zzz}]]$,
and ${\langle \zzz\rangle}^{-\bbb}\uHG{\Q}{\bbb,\bbb+1-\ccc}{\bbb-\aaa+1}{\frac{1}{\zzz}}$ in $\Q[\LL_\infty][\aaa,\bbb,\uuu]_{(\aaa-\bbb)} [[\frac{1}{\zzz}]]\cap \Q[\LL_\infty][[\aaa,\bbb,\uuu,\frac{1}{\zzz}]]$.

The above equality thus give rise to the equality
\begin{equation}\label{eq: VQ=FQ}
\begin{pmatrix}
[\V^\Q_{\vec{\infty 1}}(\zzz)]_{(1,1)}, & [\V^\Q_{\vec{\infty 1}}(\zzz)]_{(1,2)}
\end{pmatrix}
=
\begin{pmatrix} 
{\langle \zzz\rangle}^{-\aaa}\uHG{\Q}{\aaa,\aaa+1-\ccc}{\aaa-\bbb+1}{\frac{1}{\zzz}}, &
{\langle \zzz\rangle}^{-\bbb}\uHG{\Q}{\bbb,\bbb+1-\ccc}{\bbb-\aaa+1}{\frac{1}{\zzz}}
\end{pmatrix}
\end{equation}
in
$\Mat_{1\times 2}\left(\Q[\LL_\infty][[\aaa,\bbb,\uuu,\frac{1}{\zzz}]]\right)$.

In the $p$-adic case, as shown in  \cite[Theorem 3.15]{F04}, 
all coefficients of $G_{\vec{01}}^{p,\varpi}(\e_0,\e_1)(z)$ are 
described by combinations of $\Li^{p,\varpi}_{\mathbf k}(z)$'s and $\log^{p,\varpi}(z)$,
precisely in the same manner as in Proposition \ref{prop:explicit formulae} for the complex case.
Therefore the (1,1)- and (1,2)-entries of the restriction of the matrix 
$\V^{p,\varpi}_{\vec{\infty 1}}(z)$ to $]\infty[$
agree with the left hand side of \eqref{eq: VQ=FQ}.
On the other hand,
since $\uHG{p,\varpi}{\aaa,\bbb}{\ccc}{\zzz}_{]0[}$ is given by 
\eqref{eq: formal pHG at 0}, 
we see  \eqref{eq:VV infty= FF0} agrees with the right hand side of \eqref{eq: VQ=FQ}.
Whence our assertion is proved.
\end{proof}

The connection formula \eqref{eq: formal connection formula between 0 and infinity}
motivates us to define
\begin{align}\label{eq: specialization of pHG at infty}
\uHG{p, \varpi}{\alpha,\beta}{\gamma}{z}_{]\infty[} 
:=h_1 &\cdot \langle z\rangle^{-\alpha}\cdot\uHG{p, \varpi}{\alpha,\alpha+1-\gamma}{\alpha-\beta+1}{\frac{1}{z}}_{]0[}   \\ \notag
&+h_2 \cdot \langle z\rangle^{-\beta}\cdot\uHG{p, \varpi}{\beta,\beta+1-\gamma}{\beta-\alpha+1}{\frac{1}{z}}_{]0[}
\end{align}
where the first factors $h_1$ and $h_2$ are obtained by specialization of  
\eqref{eq: def eq for h1 and h2}
at $(\aaa,\bbb,\ccc)=(\alpha,\beta,\gamma)$, 
the middle factor $\langle z \rangle^{\lambda}$ with $\lambda=-\alpha,-\beta$ is defined by
\eqref{eq: lambda power of z}
and the last factors are specializations of 
$\uHG{p, \varpi}{\aaa,\aaa+1-\ccc}{\aaa-\bbb+1}{\frac{1}{\zzz}}_{]0[}$
and
$\uHG{p, \varpi}{\bbb,\bbb+1-\ccc}{\bbb-\aaa+1}{\frac{1}{\zzz}}_{]0[}$
in $\Q[\aaa,\bbb,\uuu]_{(\aaa-\bbb)}[[\frac{1}{\zzz}]]$
at $(\aaa,\bbb,\ccc,\zzz)=(\alpha,\beta,\gamma,z)$ respectively.
This definition is justified by the following theorem.

\begin{thm}\label{thm: prolongation to infty}
Let $K$  be a field extension $K$ of $\Q_p$ in $\C_p$ with $e_K<p$ and
let $z\in K$.
Suppose that  $\alpha,\beta,\gamma-1\in p\Z_p$ with $\beta\neq 0$ and
$\alpha-\beta\in\Z_{(p)}\setminus \Z $ 
and that the branch $\varpi$ of $p$-adic logarithm is chosen such that
$$
|\alpha\varpi |_p < p^{-\frac{1}{p-1}}
\quad\text{and}\quad
|\beta\varpi |_p < p^{-\frac{1}{p-1}}.
$$
Then \eqref{eq: specialization of pHG at infty}
is well-defined and 
convergent for all  $z\in K$ with
$|{z}|_p>1$. 
\end{thm}

\begin{proof}
Set $\mu=1-\gamma$, $\nu=\alpha+\beta+1-\gamma$,
$$
{X_0}=
\begin{pmatrix}
0 & \beta \\
0 & \mu
\end{pmatrix}
\qquad\text{ and }\qquad
{Y_0}=
\begin{pmatrix}
0 & 0 \\
\alpha & \nu
\end{pmatrix}
\in\Mat_2(\Q_p).
$$
By \cite[Theorem 3.30]{F04}, each coefficient of
$\Phi_{\KZ}^p(\e_0,\e_1)$  in monomials with degree (weight) $w$
can be expressed as a $\Z$-linear combination of $p$-adic multiple zeta values 
$\zeta_p({\bf k})$ with $\mathrm{wt}({\bf k})=w$.
By \cite{Ch}, we have
$\zeta_p({\mathbf k})\in {\mathfrak p}^{[\wt(\mathbf k)]}$, 
which implies the convergence of 
the substitution $\Phi_{\KZ}^p(X_0,Y_0-X_0)$  in $\GL_2(\Q_p)$.

Therefore we see that,
under the specialization 
$(\aaa,\bbb,\ccc,\zzz,\LL_\infty)=
(\alpha,\beta,\gamma,z,\log^{p,\varpi}(\frac{1}{z}))$,
the coefficients $h_1$ and $h_2$ are given by
$$
\begin{pmatrix} h_1 \\ h_2 \end{pmatrix}=
\begin{pmatrix}
1 & 1 \\
-\frac{\alpha}{\beta} & -1
\end{pmatrix}^{-1}\cdot
\Phi_{\KZ}^p(X_0,Y_0-X_0)\cdot
\begin{pmatrix}
1 & 1 \\
0 & \frac{\mu}{\beta} 
\end{pmatrix}
\begin{pmatrix} 1 \\ 0 \end{pmatrix}.
$$

Since
$\uHG{p,\varpi}{\alpha,\alpha+1-\gamma}{\alpha-\beta+1}{\frac{1}{z}}_{]0[}$ and
$\uHG{p,\varpi}{\beta,\beta+1-\gamma}{\beta-\alpha+1}{\frac{1}{z}}_{]0[}$
are convergent whenever $|z|_p>1$
as established 
by Theorem \ref{thm: convergence of pHG on 0}
due to our assumption $\alpha-\beta\in\Z_{(p)}\setminus \N_{\pm}$.
Whence the claim is proved.
%
\end{proof}

%
%


\smallskip
Thus, Theorems \ref{thm:main_result} is established 
in the case when $|z|_p>1$.
\qed
\smallskip

\subsection{Local behavior around $z=1$ and Gauss hypergeometric theorem}
\label{subsec: local behavior at 1}
We investigate the convergence of our $p$-adic hypergeometric function
in the residue class of $1$, as presented in Theorem \ref{thm:prolongation to 1} 
and prove a $p$-adic analogue of Gauss hypergeometric theorem
in Theorem \ref{thm:p-adic HG theorem}.
Our approach involves proving a connection formula in
Theorem \ref{thm: connection between 0 and 1},
which relates the points $0$ and $1$.
This argument is used to justify the validity of 
Theorems ~\ref{thm:main_result} in this setting.

The following lemma is required to prove these theorems.

\begin{lem}\label{lem: matrix for associator}
The matrix 
$\Phi_{\KZ}^p(X,-Y)\in\GL_2(\Q_p[[\aaa,\bbb,\uuu]])$
(for $\Phi_{\KZ}^p$ see \eqref{eq: Phi^p and G})
is calculated to be 
\begin{equation*}
M_{\Phi^p_\KZ}:=
\begin{pmatrix}
\GGamma{\Phi^p_\KZ}{-\uuu,-\vvv}{-\uuu-\aaa,-\uuu-\bbb}
&
\frac{\bbb}{\uuu}\left\{\GGamma{\Phi^p_\KZ}{\uuu,-\vvv}{-\aaa,-\bbb}-\GGamma{\Phi^p_\KZ}{-\uuu,-\vvv}{-\uuu-\aaa,-\uuu-\bbb}
\right\} \\
\frac{\aaa}{\vvv}\left\{\GGamma{\Phi^p_\KZ}{-\uuu,\vvv}{\aaa,\bbb}-\GGamma{\Phi^p_\KZ}{-\uuu,-\vvv}{-\uuu-\aaa,-\uuu-\bbb}\right\}
& [M_{\Phi^p_\KZ}]_{22}
\end{pmatrix}
\end{equation*}
with 
\begin{align*}
[M_{\Phi^p_\KZ}]_{22}:=&
\frac{(\aaa+\uuu)(\bbb+\uuu)}{\uuu\vvv}\GGamma{\Phi^p_\KZ}{\uuu,\vvv}{\uuu+\aaa,\uuu+\bbb} \\
&\qquad +\frac{\aaa\bbb}{\uuu\vvv}
\left\{\GGamma{\Phi^p_\KZ}{-\uuu,-\vvv}{-\uuu-\aaa,-\uuu-\bbb}-\GGamma{\Phi^p_\KZ}{-\uuu,\vvv}{\aaa,\bbb}
-\GGamma{\varphi}{\uuu,-\vvv}{-\aaa,-\bbb}
\right\},  \\
\GGamma{\Phi^p_\KZ}{s_1,s_2}{s_3,s_4}&:=\frac{\Gamma_{\Phi^p_\KZ}(s_1)\Gamma_{\Phi^p_\KZ}(s_2)}{\Gamma_{\Phi^p_\KZ}(s_3)\Gamma_{\Phi^p_\KZ}(s_4)}
\in\Q_p[[s_1,s_2,s_3,s_4]],
\end{align*}
and
$\Gamma_{\Phi^p_\KZ}(s):=\exp\left\{\sum_{n=1}^\infty 
\frac{(-1)^{n}\zeta_p(n)}{n}s^n\right\}
\in\Q_p[[s]]$.
\end{lem}

\begin{proof}
The result above is established in \cite[the proof of Theorem 3.3]{F21}
for any series in ${\Q_p}\langle\langle \e_0,\e_1\rangle\rangle$,
whose coefficients satisfy the Ohno-Zagier relations, which arise as a consequence of the regularized double shuffle relations. 
Since it is shown in  \cite{FJ} that the $p$-adic MZV's, the coefficients of
the $p$-adic KZ associator $\Phi^p_\KZ$,
satisfy these relations, the assertion follows.
\end{proof}

Since the restrictions of all $\Li^{p,\varpi}_{\bf k}(z)$'s 
and $\log^{p,\varpi}(z)$ to the residue class $]1[$ 
lie in 
$$
\Q_p[[1-z]][\log^{p,\varpi}(1-z)]\subset
\Q_p[\log^{p,\varpi}(1-z)][[1-z]],
$$
as also follows from the functional equation proved in \cite[Theorem 3.40]{F04},
the restriction of
$
\uHG{p,\varpi}{\aaa,\bbb}{\ccc}{z}
\in \mathcal O_\varpi^\col[[\aaa,\bbb,\uuu]]
$
to $]1[$
can be expressed as a series
\begin{equation}\label{eq: formal HG at 1}
\uHG{p,\varpi}{\aaa,\bbb}{\ccc}{\zzz}_{]1[}\in 
\Q_p[\LL_1][[\aaa,\bbb,\uuu,1-\zzz]].
\end{equation}
Here $\LL_1$ stands for $\log^{p,\varpi}(1-z)$.

\begin{thm}\label{thm: connection between 0 and 1}
Put $\uuu=1-\ccc$ and $\vvv=\aaa+\bbb+\uuu$.
Then the following formula holds
\begin{align}\label{eq: connection formula at 0 and 1}
\uHG{p,\varpi}{\aaa,\bbb}{\ccc}{\zzz}_{]1[}
&={\mathsf k}_1\cdot 
\uHG{p}{\aaa,\bbb}{\vvv}{1-\zzz}_{]0[} \\ \notag
&+{\mathsf k}_2\cdot (1-\zzz)\cdot{\langle 1-\zzz\rangle}^{-\vvv}\cdot 
\uHG{p}{1-\uuu-\aaa,1-\uuu-\bbb}{2-\vvv}{1-\zzz}_{]0[}
\end{align}
in $\Q_p[\LL_1][[\aaa,\bbb,\uuu,1-\zzz]]$,
where the coefficients ${\mathsf k}_1$ and ${\mathsf k}_2$ are given by
\begin{align*}
&{\mathsf k}_1=\GGamma{\Phi^p_\KZ}{-\uuu,-\vvv}{-\uuu-\aaa,-\uuu-\bbb}\in \Q_p[[\aaa,\bbb,\uuu]], \\
&{\mathsf k}_2=\frac{\aaa\bbb}{\vvv(\vvv-1)}\GGamma{\Phi^p_\KZ}{-\uuu,\vvv}{\aaa,\bbb}\in
\frac{1}{\vvv}\Q_p[[\aaa,\bbb,\uuu]],
\end{align*}
${\langle 1-\zzz\rangle}^{-\vvv}:=\exp\{-\LL_1\cdot\vvv\}
\in\Q[\LL_1][[\aaa,\bbb,\uuu,1-\zzz]]$
and the last factors are regarded to be elements in 
$\Q_p[\aaa,\bbb,\uuu]_{(\vvv-1)}[[1-\zzz]]$
by \eqref{eq: formal pHG at 0}.
\end{thm}

\begin{proof}
Our proof goes the same way as that of Theorem \ref{thm: formal connection formula between 0 and infinity}.
Similarly to \eqref{eqC} in the complex case,
we consider
\begin{align*}
&\V^{p,\varpi}_{\vec{10}}(z):=
G^{p,\varpi}_{\vec{10}}(X,-Y)(z)\cdot
\begin{pmatrix}
1 & 0 \\
\frac{-\aaa}{\vvv} & \frac{\vvv-1}{\bbb}
\end{pmatrix}
\in \GL_2(\PP[\frac{1}{\bbb\vvv}]).
\end{align*}
By \eqref{eq: pB} and \eqref{eq: Phi^p and G}, we have 
$$
G^{p,\varpi}_{\vec{01}}(\e_0,\e_1)(z)=G^{p,\varpi}_{\vec{10}}(\e_0,\e_1)(z)\Phi_{\KZ}^p(\e_0,\e_1)
$$
$$
G^{p,\varpi}_{\vec{10}}(\e_0,\e_1)(z)=G^{p,\varpi}_{\vec{01}}(\e_1,\e_0)(1-z).
$$
Therefore we obtain the transformation:
\begin{equation}\label{eq:V0110}
\V^{p,\varpi}_{\vec{01}}(z)  =\V^{p,\varpi}_{\vec{10}}(z)\cdot
{\mathsf K}
\in \GL_2(\PP[\frac{1}{\bbb}]),
\end{equation}
\begin{equation*}\label{eq: V 10 G 01} 
\V^{p,\varpi}_{\vec{10}}(z)  =G^{p,\varpi}_{\vec{01}}(-Y,X)(1-z)\cdot
\begin{pmatrix}
1 & 0 \\
-\frac{\aaa}{\vvv} &\frac{\vvv-1}{\bbb}
\end{pmatrix}\in \GL_2(\PP[\frac{1}{\bbb\vvv}]),
\end{equation*}
where
\begin{equation*}\label{eq: K}
{\mathsf K}:=
\begin{pmatrix}
1 & 0 \\
\frac{\aaa\bbb}{\vvv (\vvv-1)} &\frac{\bbb}{\vvv-1}
\end{pmatrix}
\cdot\Phi^p_\KZ(X,-Y)\cdot
\begin{pmatrix}
1 & 1 \\
0 &\frac{\uuu}{\bbb}
\end{pmatrix}\in \GL_2((\Q_p[[\aaa,\bbb,\uuu]][\frac{1}{\bbb\vvv}])
\end{equation*}
(we note that  $(\vvv-1)^{-1}\in\Q_p[[\aaa,\bbb,\uuu]]$).
By Lemma \ref{lem: matrix for associator}, 
it is calculated to be
\begin{equation*}\label{eq: def eq for k1 and k2}
\begin{pmatrix} {\mathsf k}_1 \\ {\mathsf k}_2 \end{pmatrix}=
{\mathsf K}
\begin{pmatrix}
1 \\
0 \\
\end{pmatrix}
\end{equation*}

By comparing the $(1,1)$-entry of both sides of \eqref{eq:V0110},
we obtain
\begin{equation}\label{eq:pre connection formula z and 1-z}
\uHG{p,\varpi}{\aaa,\bbb}{\ccc}{z}=
{\mathsf k}_1\cdot
[\V^{p,\varpi}_{\vec{10}}(z)]_{(1,1)}
+
{\mathsf k}_2\cdot
[\V^{p,\varpi}_{\vec{10}}(z)]_{(1,2)}
\end{equation}
in $\PP[\frac{1}{\bbb\vvv(\vvv-1)}]$.
By restricting \eqref{eq:pre connection formula z and 1-z}
to $]\infty[$ and combining
it with Lemma \ref{lem: formal connection formula between 0 and 1} below, we derive
the connection formula
\eqref{eq: connection formula at 0 and 1}.
\end{proof}

\begin{lem}\label{lem: formal connection formula between 0 and 1}
The (1,1)- and (1,2)-entries of the restriction of the matrix $\V^{p, \varpi}_{\vec{10}}(z)$ to  $]1[$
are given by 
\begin{equation}\label{eq:VV 1= FF0}
\begin{pmatrix} 
\uHG{p,\varpi}{\aaa,\bbb}{\aaa+\bbb+1-\ccc}{1-z}_{]1[},
& (1-z)\cdot{\langle 1-\zzz\rangle}^{-\vvv}\uHG{p,\varpi}{\ccc-\aaa,\ccc-\bbb}{\ccc-\aaa-\bbb+1}{1-z}_{]1[}
\end{pmatrix}.
\end{equation}
\end{lem}

\begin{proof}
The proof proceeds in a manner analogous to that of
Lemma \ref{lem: formal connection formula between 0 and infinity}.
In the complex case, by \eqref{eqB} and \eqref{eqC}, we have
$$
\V^\C_{\vec{10}}(z)=
G^\C_{\vec{01}}(-Y_0,X_0)(1-z)
\cdot
\begin{pmatrix}
1 & 0 \\
\frac{-a}{v} & \frac{v-1}{b}
\end{pmatrix}.
$$
By \eqref{eqC} and \eqref{eqD}
the first row of this matrix is computed as
\begin{align*}
&\begin{pmatrix}
[\V^\C_{\vec{10}}(z)]_{(1,1)}, & [\V^\C_{\vec{10}}(z)]_{(1,2)}
\end{pmatrix} \\
&\qquad=
\begin{pmatrix} 
\uHG{\C}{a,b}{a+b+1-c}{1-z}, & (1-z)\cdot (1-z)^{c-a-b-1}\uHG{\C}{c-a,c-b}{c-a-b+1}{1-z}\end{pmatrix}
\end{align*}
for $(a,b,c)\in\C^3$ under appropieate analytic conditions.
Upon interpreting $a$, $b$, $c-1$ and ${1-z}$ as formal variables
and employing the formulas in Proposition \ref{prop:explicit formulae},
we see that $[\V^{\C}_{\vec{10}}(z)]_{(1,1)}$ naturally determines an element,
denoted by
$$[\V^{\Q}_{\vec{10}}(\zzz)]_{(1,1)} \in 
\frac{1}{\vvv}\Q[\LL_1][[\aaa,\bbb,\uuu,1-\zzz]]
$$
where $\LL_1$ corresponds to $\log(1-z)$.
Similarly by the expansion \eqref{def eq: HG},
$\uHG{\C}{a,b}{a+b+1-c}{1-z}$
determines an element,
denoted by 
$$
\uHG{\Q}{\aaa,\bbb}{\aaa+\bbb+1-\ccc}{1-\zzz}_{]1[}
\in\Q[\aaa,\bbb,\uuu]_{(\vvv-1)}[[1-\zzz]]
\cap \Q[[\aaa,\bbb,\uuu,1-\zzz]].
$$
Analogously we have elements 
$[\V^{\Q}_{\vec{10}}(\zzz)]_{(1,2)}$ in 
$\frac{1}{\bbb}\Q[\LL_1][[\aaa,\bbb,\uuu,1-\zzz]]$ and
$$
(1-\zzz)\cdot {\langle 1-\zzz\rangle}^{-\vvv}\uHG{p,\varpi}{\ccc-\aaa,\ccc-\bbb}{\ccc-\aaa-\bbb+1}{1-\zzz}_{]1[}
$$
in
$(1-\zzz)\cdot \Q[\LL_1][[\vvv]]\cdot \left(\Q[\aaa,\bbb,\uuu]_{(\vvv-1)}[[1-\zzz]]
\cap \Q[[\aaa,\bbb,\uuu,1-\zzz]]\right)$ which is a subspace of
$(1-\zzz)\cdot \Q[\LL_1][[\aaa,\bbb,\uuu,1-\zzz]]$.

The above equality  gives rise to
\begin{align}\label{eq: VV1= FF0}
&\begin{pmatrix}
[\V^\Q_{\vec{10}}(\zzz)]_{(1,1)}, & [\V^\Q_{\vec{10}}(\zzz)]_{(1,2)}
\end{pmatrix} \\ \notag
&= 
\begin{pmatrix} 
\uHG{\Q}{\aaa,\bbb}{\aaa+\bbb+1-\ccc}{1-\zzz}, & (1-\zzz)\cdot {\langle 1-\zzz\rangle}^{-\vvv}\uHG{\Q}{\ccc-\aaa,\ccc-\bbb}{\ccc-\aaa-\bbb+1}{1-\zzz}\end{pmatrix}
\end{align}
in
$\Mat_{1\times 2}\left(\Q[\LL_1][[\aaa,\bbb,\uuu,1-\zzz]]\right)$.

Again as shown in  \cite[Theorem 3.15]{F04}, 
all coefficients of $G_{\vec{01}}^{p,\varpi}(\e_0,\e_1)(z)$ are 
described by combinations of $\Li^{p,\varpi}_{\mathbf k}(z)$'s and $\log^{p,\varpi}(z)$,
precisely in the same manner as in Proposition \ref{prop:explicit formulae} for the complex case.
Therefore the (1,1)- and (1,2)-entries of the restriction of the matrix 
$\V^{p,\varpi}_{\vec{10}}(z)$ to $]1[$
agree with the left hand side of \eqref{eq: VV1= FF0}.
On the other hand,
since $\uHG{p,\varpi}{\aaa,\bbb}{\ccc}{\zzz}_{]1[}$ is given by 
\eqref{eq: formal pHG at 0}, 
we see  \eqref{eq:VV 1= FF0} agrees with the right hand side of \eqref{eq: VV1= FF0}.
Whence our assertion is proved.
\end{proof}

Similarly to \eqref{eq: specialization of pHG at infty},
the connection formula \eqref{eq: connection formula at 0 and 1} motivates us to define 
\begin{align}\label{eq: formal pHG z and 1-z}
&\uHG{p,\varpi}{\alpha,\beta}{\gamma}{z}_{]1[}:= 
\GGamma{\Phi^p_\KZ}{-\mu,-\nu}{-\mu-\alpha,-\mu-\beta}\cdot
\uHG{p,\varpi}{\alpha,\beta}{\nu}{1-z}_{]0[} \\ \notag
&\quad+
\frac{\alpha\beta}{\nu(\nu-1)}
\GGamma{\Phi^p_\KZ}{\mu,\nu}{\alpha,\beta}\cdot
(1-z) \cdot\langle 1-z\rangle^{-\nu}\cdot\uHG{p,\varpi}{1-\alpha-\mu,1-\beta-\mu}{2-\nu}{1-z}_{]0[},
\end{align}
where $\uHG{p,\varpi}{\aaa,\bbb}{\vvv}{1-\zzz}_{]0[}$ and
$\uHG{p,\varpi}{1-\aaa-\uuu,1-\bbb-\uuu}{2-\vvv}{1-\zzz}_{]0[}$
are regarded to be elements in 
$\Q[\aaa,\bbb,\uuu]_{(\vvv-1)}[[1-\zzz]]\cap 
\frac{1}{\vvv}\Q
[[\aaa,\bbb,\uuu,1-\zzz]]$.
This definition is justified by the following theorem.

\begin{thm}\label{thm:prolongation to 1}
Let $z$ be in a field extension $K$ of $\Q_p$ in $\C_p$ with $e_K<p-1$
where $e_K$ means  the absolute ramification index  of $K$.
Assume that $\alpha,\beta,\gamma\in\Z_{p}$ with 
$|\alpha|_p, |\beta|_p, |\gamma-1|_p<1$ and
$\alpha+\beta-\gamma\in\Z_{(p)}\setminus\N_{\pm}$
and that the branch $\varpi$ of $p$-adic logarithm is chosen such that
$\varpi\in\mathcal O_{K}$. 
Then \eqref{eq: formal pHG z and 1-z}
is well-defined and 
convergent for all  $z\in K$ with $0<|z-1|_p<1$.
\end{thm}

\begin{proof}
By our assumption on $\varpi$ 
and $|\nu|_p<1$
($\nu=\alpha+\beta+1-\gamma$),
we see that
$\langle 1-z\rangle^{-\nu}:=\exp\{-\nu\log^{p,\varpi} (1-z)\}$ 
converges  by Lemma \ref{lem: power map convergence}.

By the conditions  $\alpha,\beta,\gamma\in\Z_{p}$  and $\alpha+\beta-\gamma\in\Z_{(p)}\setminus\N_\pm$,
we learn from 
Theorem \ref{thm: convergence of pHG on 0}.(2),
that both $p$-adic series 
$\uHG{p,\varpi}{\alpha,\beta}{\alpha+\beta+1-\gamma}{1-z}_{]0[}$ 
and
$(1-z)\cdot\langle 1-z\rangle^{-\nu}\uHG{p,\varpi}{\gamma-\alpha,\gamma-\beta}{\gamma-\alpha-\beta+1}{1-z}_{]0[}$
converge when $|z-1|_p<1$.

Again by our condition $|\alpha|_p, |\beta|_p, |\gamma-1|_p<1$
and the fact $\zeta_p(k)\in\Z_p$ for all $k$ which follows from \cite{Ch},
both $\GGamma{\Phi^p_\KZ}{-\mu,-\nu}{-\mu-\alpha,-\mu-\beta}$
and
$\GGamma{\Phi^p_\KZ}{\mu,-\nu}{-\alpha,-\beta}$
converge on $\Q_p$.
Therefore we obtain the claim. 
\end{proof}

We note that  in the complex case it is calculated to be
$$\Gamma_{\varPhi^\C_\KZ}(z)=\exp\{\sum_{n=2}^\infty \frac{(-1)^n\zeta(n)}{n}z^n\}
=e^{\gamma z}\Gamma(1+z),$$ 
Formula \eqref{eq: formal pHG z and 1-z}
is regarded as a $p$-adic analogue of the classical connection formula (cf. \cite[15.3.7]{AS})
\begin{align*}
\uHG{\C}{a,b}{c}{z}=
&\frac{\Gamma(c)\Gamma(c-a-b)}{\Gamma(c-a)\Gamma(c-b)}
\uHG{\C}{a,b}{a+b-c+1}{1-z}\\ \notag
&+ (1-z)^{c-a-b}
\frac{\Gamma(c)\Gamma(a+b-c)}{\Gamma(a)\Gamma(b)}
\uHG{\C}{c-a,c-b}{c-a-b+1}{1-z}.
\end{align*}
 
The following is a $p$-adic analogue of Gauss hypergeometric theorem
\eqref{eq: hypergeometric equation}.

\begin{thm}\label{thm:p-adic HG theorem}
Under the assumption of Theorem \ref{thm:prolongation to 1},
we have
$$
\lim_{\substack{ z\to 1 \\ z\in K }}\uHG{p,\varpi}{\alpha,\beta}{\gamma}{z}_{]1[}
=\prod_{k=1}^\infty\frac{\Gamma_p(1+p^k\mu+p^k\alpha)\cdot\Gamma_p(1+p^k\mu+p^k\beta)}
{\Gamma_p(1+p^k{\mu})\cdot\Gamma_p(1+p^k{\nu})}
$$
with $\mu=1-\gamma$ and $\nu=\alpha+\beta+1-\gamma$,
where $\Gamma_p$ is Morita's  $p$-adic gamma function  (\cite{M}) defined by 
\begin{equation}\label{eq: Morita gamma}
\Gamma_p(1-z)=\exp\left\{\gamma_pz+\sum_{n=2}^\infty L_p(n,\omega^{1-n})\frac{z^n}{n}\right\}
\end{equation}
($\gamma_p$: a $p$-adic analogue of Euler constant, cf. \cite{S})
converging on $|z|_p<1$.
\end{thm}

\begin{proof}
When  writing $1-z=\epsilon\cdot p^r (1+m)$
with 
$\epsilon\in\mu_\infty$,
$r\in\frac{1}{e_K}\N$ and $m\in\mathfrak m_{K}$,
we have
\begin{align*}
\lim_{\substack{ z\to 1 \\ z\in K }}(1-z)\cdot\langle 1-z\rangle^{-\nu}
&=\lim_{r\to\infty}(1-z)\cdot
\exp\{-\nu r \varpi-\nu\log^p(1+m)) \} \\
&=\lim_{r\to\infty} \epsilon p^r(1+m)\cdot
\exp\{-\nu r \varpi\}\exp\{-\nu\log^p(1+m)) \} 
=0.
\end{align*}
Here we use the results 
$|\log^p(1+x)|_p=|x|_p$ and
$|\exp (x)-1|_p<1$ when $|x|_p<p^{\frac{-1}{p-1}}$ 
shown in \cite[Lemma 5.5 and the proof of Proposition 5.7]{W}.
So we have
$$
 \lim_{\substack{ z\to 1 \\ z\in K }}(1-z)\cdot\langle 1-z\rangle^{-\nu}\cdot
 \uHG{p}{\gamma-\alpha,\gamma-\beta}{\gamma-\alpha-\beta+1}{1-z}=0.
$$
By \eqref{eq: formal pHG z and 1-z},
we have
$$
\lim_{\substack{ z\to 1 \\ z\in K }}\uHG{p,\varpi}{\alpha,\beta}{\gamma}{z}_{]1[}
=\GGamma{\Phi^p_\KZ}{-\mu,-\nu}{-\mu-\alpha,-\mu-\beta}.
$$
By
$$
\Gamma_{\Phi^p_\KZ}(z)=
\exp\{\sum_{n=2}^\infty\frac{(-1)^n}{n}\zeta_p(n)z^n\}
=\exp\{\frac{p\gamma_p}{p-1}z\}\prod_{k=1}^\infty\Gamma_p(1+p^kz)^{-1},
$$
and $\Gamma_{\Phi^p_\KZ}(z)\Gamma_{\Phi^p_\KZ}(-z)=1$
which can be deduced from the fact $\zeta_p(2k)=$ for all $k>0$, 
we obtain the claim.
\end{proof}

\smallskip
Thus, Theorems \ref{thm:main_result} is established 
in the case when $0<|z-1|_p<1$.
\qed
\smallskip

\begin{rem}
Another construction of a $p$-adic hypergeometric function was developed by Dwork (\cite{Dw}).
However, at present, the precise relationship between our 
$p$-adic hypergeometric function
$\uHG{p}{\alpha,\beta}{\gamma}{z}$ ($\alpha, \beta, \gamma\in\Z_{(p)}$ 
with $\gamma\neq 0,-1,-2,\dots$)
and Dwork's framework remains unclear,
and further investigation is needed to clarify how these two constructions are related.
\end{rem}


\begin{thebibliography}{Utah}

\bibitem[AAR]
{AAR}
Andrews, G. E., Askey, R., and Roy, R.;
{\it Special functions},
Encyclopedia of Mathematics and its Applications, {\bf 71}, 
Cambridge University Press, Cambridge, 1999.

\bibitem[AS]
{AS}
Abramowitz, M. and  Stegun, I. A.;
{\it Handbook of mathematical functions with formulas, graphs, and mathematical tables}, 
Reprint of the 1972 edition. Dover Publications, Inc., New York, 1992. {\rm xiv}+1046

\bibitem[Ber]{Ber}
P. \ Berthelot, 
\textit{Cohomologie rigide et cohomologie rigide \`{a} support propre}, 
Premi\`{e}re partie,                            
Pr\'{e}publication IRMAR 96-03, 89 pages (1996).


\bibitem[Bes]{Bes}
A.\ Besser,
\textit{Coleman integration using the Tannakian formalism}, 
Math. Ann. \textbf{322} (2002) 1, 19--48. 

\bibitem[BGR]{BGR}
S.\ Bosch, U. G\"{u}ntzer, and R. \ Remmert,
\textit{Non-Archimedean analysis. A systematic approach to rigid analytic geometry}, 
Grundlehren der Mathematischen Wissenschaften, {\bf 261}, Springer-Verlag, Berlin.


\bibitem[Ch]
{Ch}
Chatzistamatiou, A.;
{\it On integrality of $p$-adic iterated integrals}, 
J. Algebra 474 (2017), 240--270.

\bibitem[Co]
{Col}
Coleman, R.;
{\it Dilogarithms, regulators and $p$-adic $L$-functions},
Invent. Math. 69 (1982), no. 2, 171--208. 


\bibitem[Dr]
{Dr89}
Drinfeld, V. G.,
{\it  Quasi-Hopf algebras and Knizhnik-Zamolodchikov equations},
Problems of modern quantum field theory (Alushta, 1989), 1--13, 
Res. Rep. Phys., Springer, Berlin, 1989. 



\bibitem[Dw]
{Dw}
Dwork, B.;
{\it $p$-adic cycles},
Inst. Hautes \'{E}tudes Sci. Publ. Math. No. {\bf 37} (1969), 27--115. 

\bibitem[F03]
{F03}
Furusho, H.; 
{\it The multiple zeta value algebra and the stable derivation algebra},
Publ. Res. Inst. Math. Sci. {\bf 39}. no 4. (2003). 695--720.

\bibitem[F04]
{F04}
Furusho, H.; 
{\it  $p$-adic multiple zeta values I -- $p$-adic multiple polylogarithms and the $p$-adic KZ equation}, 
Invent. Math. {\bf 155} (2004), no. 2, 253--286. 

\bibitem[F21]
{F21} 
Furusho, H.;
{\it The $\ell$-adic hypergeometric function and associators},
Tunis. J. Math. {\bf 5} (2023), no. 1, 1--29.

\bibitem[FJ]
{FJ}
Furusho, H. and  Jafari, A.;
{\it  Regularization and generalized double shuffle relations for $p$-adic multiple zeta values},
Compos. Math. {\bf 143} (2007), no. 5, 1089--1107.

\bibitem[K]
{Ke}
Kedlaya, ~K. ~S.;
{\it $p$-adic differential equations},
Cambridge Studies in Advanced Mathematics, 125. Cambridge University Press, Cambridge, 2010. 


\bibitem[M]{M}
Morita, Y.; 
A $p$-adic integral representation of the $p$-adic $L$-function. 
J. Reine Angew. Math. 302 (1978), 71--95. 


\bibitem[OZ]{OZ}
Ohno, ~Y. and Zagier ~Don.;
{\it Multiple zeta values of fixed weight, depth, and height},
Indag. Math. (N.S.)  {\bf 12} (2001), no. 4, 483--487.

\bibitem[Oi]{O}
Oi, S.;
{\it Gauss hypergeometric functions, multiple polylogarithms, and multiple zeta values}, Publ. Res. Inst. Math. Sci. {\bf 45} (2009), no. 4, 981--1009.

\bibitem[S]{S}
Schikhof, ~W. ~H.;
{\it Ultrametric calculus;
An introduction to p-adic analysis}
Cambridge Stud. Adv. Math., {\bf 4},
Cambridge University Press, Cambridge, 1984.

\bibitem[W]{W} 
Washington, L.\,C.;
{\it Introduction to {C}yclotomic {F}ields}, 
{Second} edition, {Graduate Texts in Mathematics}, {\bf 83}, {Springer-Verlag},
{New York}, {1997}.


\end{thebibliography}
\end{document}